\documentclass[12pt,reqno]{amsproc}
\usepackage[T1]{fontenc}

\usepackage{amssymb}
\usepackage{amsmath}
\usepackage{amsthm}
\usepackage{amsfonts}
\usepackage{fullpage}
\usepackage{enumerate}
\usepackage{tikz-cd} 
\usepackage{comment}
\usepackage{float}
\usepackage{verbatim}
\usepackage{enumitem}
\usepackage{fancyhdr}
\usepackage{todonotes}
\usepackage[foot]{amsaddr}
\usepackage{svg}
\usepackage{multirow}
\usepackage{url}
\usepackage{textcomp}

\setlist[enumerate]{label*=\arabic*.} 

\usepackage[hidelinks]{hyperref}
\hypersetup{
    colorlinks,
    citecolor=magenta,
    filecolor=magenta,
    linkcolor=red,
    urlcolor=blue
}

\newcommand{\nc}{\newcommand}
\newcommand{\rc}{\renewcommand}

\nc\on{\operatorname}
\nc\ol{\overline}
\nc{\ul}{\underline}
\nc{\wt}{\widetilde}
\nc{\us}{\underset}
\nc{\os}{\overset}
\nc{\wh}{\widehat}

\nc{\Spec}{\on{Spec}}
\nc{\Proj}{\on{Proj}}
\nc\Id{\on{Id}}
\nc{\Hom}{\on{Hom}}
\nc{\res}{\on{res}}
\nc{\Cone}{\on{Cone}}
\nc{\Conv}{\on{Conv}}
\nc{\tr}{\on{tr}}
\nc{\Lie}{\on{Lie}}
\nc{\Der}{\on{Der}}
\nc{\Aut}{\on{Aut}}
\nc{\End}{\on{End}}
\nc{\Ad}{\on{Ad}}
\nc{\ad}{\on{ad}}
\nc{\Ext}{\on{Ext}}
\nc{\Mor}{\on{Mor}}
\nc{\Tor}{\on{Tor}}

\nc{\tcap}{\bar\pitchfork}
\nc{\del}{\nabla}
\renewcommand{\d}{\text{d}}

\nc{\la}{\langle}
\nc{\ra}{\rangle}
\nc{\rt}{\sqrt}

\nc{\injto}{\hookrightarrow}
\nc{\surjto}{\twoheadrightarrow}
\nc{\mapsfrom}{\mathrel{\reflectbox{\ensuremath{\mapsto}}}}
\nc{\from}{\mathrel{\reflectbox{\ensuremath{\to}}}}
\nc{\iv}{^{-1}}

\nc{\acts}{\curvearrowright}
\nc{\dto}{\dashrightarrow}

\nc{\di}{\d_{x_i}}
\rc{\dj}{\d_{x_j}}

\rc{\b}{\bullet}
\mathchardef\hy="2D
\nc{\sm}{\setminus}

\nc{\dol}{\partial}

\nc{\pder}[1]{\frac{\dol}{\dol {#1}}}
\nc{\pderh}[2]{\frac{\dol^{#2}}{\dol {#1}^{#2}}}

\nc{\bs}{\bigskip}
\nc{\ms}{\medskip}

\nc{\noi}{\noindent}
\nc{\tn}{\textnormal}
\nc{\tb}{\textbf}
\nc{\mb}{\mathbb}
\nc{\mc}{\mathcal}
\nc{\dsp}{\displaystyle}
\nc{\tc}{\textcolor}

\nc{\CC}{\mathbb{C}}
\nc{\FF}{\mathbb{F}}
\nc{\NN}{\mathbb{N}}
\nc{\PP}{\mathbb{P}}
\nc{\QQ}{\mathbb{Q}}
\nc{\RR}{\mathbb{R}}
\nc{\ZZ}{\mathbb{Z}}

\nc{\CL}{\mathcal{L}}
\nc{\CE}{\mathcal{E}}
\nc{\CF}{\mathcal{F}}
\nc{\CG}{\mathcal{G}}
\nc{\CH}{\mathcal{H}}
\nc{\CO}{\mathcal{O}}

\nc{\mf}{\mathfrak}
\rc{\sf}{\mathsf}

\nc{\GL}{\on{GL}}

\nc{\vb}[1]{\vec{\bold{#1}}}

\nc{\vv}[2]{\begin{bmatrix} #1\\ #2 \end{bmatrix}}
\nc{\vvv}[2]{\begin{bmatrix} #1\\ #2 \end{bmatrix}}
\nc{\mat}[4]{\begin{bmatrix} #1 & #2 \\ #3 & #4 \end{bmatrix}}
\nc{\mmat}[9]{\begin{bmatrix} #1 & #2 & #3 \\ #4 & #5 & #6 \\ #7 & #8 & #9 \end{bmatrix}}
\nc{\ptl}[2]{\frac{\d #1}{\d #2}}

\theoremstyle{definition}

\rc{\theq}{\thesubsection.\Alph{q}}

\rc{\thedis}{\thesubsection.\Alph{dis}}
\newtheorem{theorem}{Theorem}[section]
\newtheorem{lemma}{Lemma}[section]
\newtheorem{definition}{Definition}[section]
\newtheorem{proposition}{Proposition}[section]
\newtheorem{corollary}{Corollary}[section]

\begin{document}

\title{Harmonic Forms, Hodge Theory and \\ the Kodaira Embedding Theorem}
\author[U. Lim]{Uzu Lim}
\address{Mathematical Institute, 
University of Oxford, Radcliffe Observatory, Andrew Wiles Building, Woodstock Rd, Oxford OX2 6GG}
\email{lims@maths.ox.ac.uk}

\begin{abstract}
    In this expository article, we outline the theory of harmonic differential forms and its consequences. We provide self-contained proofs of the following important results in differential geometry: (1) Hodge theorem, which states that for a compact complex manifold, the de Rham cohomology group is isomorphic to the group of harmonic forms, (2) Hodge decomposition theorem, which states that for a K\"ahler manifold, the de Rham cohomology group decomposes into the Dolbeault cohomology groups, and (3) The Kodaira Embedding theorem, which gives a criterion of when a compact complex manifold is in fact a smooth complex projective variety. The basic theory of vector bundles is also contained for completeness.
\end{abstract}

\maketitle

\section{Introduction}

A harmonic differential form is a differential form satisfying $\Delta \varphi = 0$ for the Laplacian $\Delta$. It turns out that there is a unique harmonic representative of each de Rham cohomology class. Harmonic representatives also form an abelian group that is isomorphic to the de Rham cohomology group:
$$\mc H^r(X) \cong H^r(X,\CC)$$
By using $\mc H^r(X)$, we simplify many discussions as we can just work with a single differential form rather than a whole ensemble of them. Identifying such differential form is viewed as a projection within a Banach space, and thus we will need some functional analysis to establish such identification. 

One application of this theory is the Kodaira embedding theorem, a powerful theorem that characterizes smooth complex projective varieties among complex manifolds. The theorem says that a complex manifold is isomorphic to a smooth complex projective variety iff the manifold possesses a \textit{positive line bundle}. Sections of the line bundle is used to give a map to a projective space, and the fact that the map is indeed an embedding can be reinterpreted as a surjectivity relation on global sections of some vector bundles. The surjectivity relation can in turn be interpreted as vanishing of certain first sheaf cohomology groups. This vanishing of sheaf cohomology groups is implied by the general result of Kodaira and Nakano, which uses harmonic differential forms.

This article is divided into three sections. In the first section, we introduce the geometry of vector bundles and sheaf theory pertaining to it. Major tools introduced include metric, connection, curvature, and Chern class on a vector bundle. Along the way, we also prove de Rham theorem and Dolbeault theorem using acyclic sheaf resolutions. In the second section, we discuss Hodge theory. First we discuss general differential operators, then elliptic operators and their parametrix. Then we prove the Hodge theorem ($\mc H^r(X) \cong H^r(X,\CC)$) and the Hodge decomposition theorem ($\bigoplus H^{p,q}(X) \cong H^r(X,\CC)$) for compact K\"ahler manifolds. In the third section, we discuss positive line bundles and Kodaira embedding theorem.

The primary references for this article are \cite{wells} and \cite{griffith harris}. For an introductory account of differential geometry, one might consult \cite{lee riemannian}.

\pagebreak
\tableofcontents

\pagebreak
\section{Vector Bundles}

We first introduce the language of vector bundles and related concepts such as divisors, connection, curvature, and Chern class.

\subsection{Basic Properties}

A \emph{real vector bundle} of rank $k$ is a topological space $E$ equipped the following data:
\begin{enumerate}
    \item A continuous surjection $\pi:E \rightarrow X$
    \item Real vector space structure on each fiber $\pi^{-1}(x) \cong \RR^k$
    \item (Local trivialization) For each $x \in X$, there is an open neighborhood $U$ and a homeomorphism $\varphi_U: U \times \RR^k \rightarrow \pi^{-1}(U)$ which respects fiber: $(\pi\circ \varphi_U)(x,v)=x$ and the induced map on fiber $\RR^k \rightarrow \pi^{-1}(x)$ given by $v \mapsto \varphi(x,v)$ is a linear isomorphism.
\end{enumerate}
In short, it's a topological space that is locally a product of the underlying space with a real vector space. We can analogously define complex vector bundles. 

A \emph{smooth vector bundle} is a real vector bundle where in the above definition, $E,X$ are smooth manifolds, $\pi$ is smooth, and local trivialization maps $\varphi_U$ are diffeomorphisms. We analogously define \emph{holomorphic vector bundle} to be a complex vector bundle with smooth / diffeomorphism replaced by complex/biholomorphic.

The data of a smooth/holomorphic vector bundle on a smooth/complex manifold $X$ is equivalent to giving an open cover $\mc U = \{ U_\alpha\}$ and smooth/holomorphic transition maps $g_{\alpha \beta} : U_{\alpha} \cap U_\beta \rightarrow \GL_k\RR$ such that $g_{\alpha \beta}(x) \circ g_{\beta \gamma}(x) \circ g_{\gamma \alpha}(x) = \on{id}_{\RR^k}$. It's easy to check that this condition is satisfied given a bundle. Conversely given such data, we can glue disjoint union of trivial bundles on $U_\alpha$ using transition functions.

A \emph{section} of a vector bundle $E\rightarrow X$ is a continuous map $s:X \rightarrow E$ such that $\pi \circ s = \on{id}_X$ (i.e. each point is mapped to a vector in the fiber of the point). A \emph{smooth} (resp. \emph{holomorphic}) section of a smooth bundle is a section that is a smooth (resp. holomorphic) map.

If $U \subset X$ is Euclidean ($U \cong V \subseteq \RR^n$) and $E$ is trivial over $U$ ($\pi^{-1}(U) \cong U \times \RR^k$), then a smooth section on $U$ can be seen as a smooth map $V \rightarrow V \times \RR^k$.

Given a smooth manifold $X$, define \emph{stalk} at each $x \in X$ by
$$\mathcal E_{X,x} := \lim_\rightarrow \mathcal E_X(U)$$
where direct limit is taken over neighborhoods $U$ of $x$ and the direct system is given by restriction maps of smooth functions. Here, $[f]=[g]$ iff $f$ and $g$ agree locally. Thus each element (called \emph{germ}) of $\mathcal E_{X,x}$ represents values of a smooth function at an infinitesimal neighborhood of $x$.

Given smooth vector bundles $E,F$, we can construct more bundles out of them, namely
\begin{align*}
    E \oplus F, E \otimes F, \wedge^p E, E^*
\end{align*}
They are defined by fiberwise operation (e.g. fiber of $E \oplus F$ at $x$ is $E_x \oplus F_x$) and transition functions. Suppose $E,F$ are vector bundles over $X$ of rank $k,l$ that trivialize over $\{U_\alpha\}$. Then over $U_{\alpha \beta}$, we can give transition function for $E \oplus F, E\otimes F, \wedge^p E, E^*$ in the obvious way.

In defining bundles like $E \oplus F$ from $E,F$, we can use the following lemma to topologize $E \oplus F$.
\begin{lemma}
    Suppose $X$ is a set and $\{A_\alpha \}_{\alpha \in I}$ is an open cover. Suppose $A_\alpha'$ are topological spaces with bijection $f_\alpha : A_\alpha' \rightarrow A_\alpha$. Then $X$ has a topology that makes every $f_\alpha$ a homeomorphism iff for each $\alpha, \beta \in I$, the topology induced by $f_\alpha$ and $f_\beta$ on $A_{\alpha} \cap A_{\beta}$ are identical.
\end{lemma}
\begin{proof}
    If such a topology on $X$ exists and the topology induced by $f_{\alpha}, f_\beta$ on $A_\alpha \cap A_\beta$ aren't the same, then $\exists U \subseteq A_\alpha \cap A_\beta$ so that $f_\alpha^{-1}(U)$ or $f_\beta^{-1}(U)$ aren't open, and both cases contradict $f_\alpha, f_\beta$ being homeomorphisms.
    
    Conversely if topology induced by $f_\alpha, f_\beta$ coincide on $A_\alpha \cap A_\beta$, then consider the topology on $X$ given by taking images of open sets under $f_\alpha$ as the subbasis. Suppose $U$ is an open set in $X$ and $U \subseteq A_\alpha$. Then we may write $U$ as:
    \begin{align*}
        U = \bigcup_j \bigcap_k f_{\beta_{j,k}}(U_{j,k}) = \bigcup_j \bigcap_k (f_{\beta_{j,k}}(U_{j,k}) \cap A_\alpha) = \bigcup_j \bigcap_k f_\alpha(V_{j,k}) = f_\alpha(V)
    \end{align*}
    Here the intersection is understood to be finite, and we used the fact that a set of the form $f_\beta(V)\cap A_\alpha$ can be rewritten as $f_\alpha(V')$ for some $V' \subseteq A_\alpha'$ by the hypothesis. Thus each open subset of $A_\alpha$ is an image of an open set under $f_\alpha$. Conversely every image of an open subset of $A_\alpha'$ is open by definition of the topology. Thus we're done.
\end{proof}

\medskip
\textbf{Example: Tangent Bundle.}

First we define tangent space. Given a smooth $n$-manifold $X$, the \emph{tangent space} $T_x X$ at $x \in X$ is defined as the set of all \emph{derivations}, which are maps $\mc E_{X,x} \rightarrow \RR$ that satisfy Leibniz rule $D(fg) = D(f)g(p) + D(g)f(p)$ (it admits an obvious real vector space structure). It can be shown that, given local coordinate $x$, the set $\{\frac{\partial}{\partial x_1}, \cdots \frac{\partial}{\partial x_n}\}$ is a basis of $T_x X$, and this gives a more vector-like intuition to a derivation. Note that definition of a tangent vector as a derivation is basis-free.

Given these tangent spaces, first let $TX = \bigsqcup_x T_xX$ set-theoretically. Then for each subset $U \subseteq X$ with $x: U \cong V \subseteq \RR^n$, let chart of $TX$ at $\pi^{-1}(U)$ be
\begin{align*}
    & \varphi_U : \pi^{-1}(U) \rightarrow U \times \RR^n \\
    & \varphi_U : v \mapsto (\pi(v), (a_1, \cdots a_n))
\end{align*}
where local expression of derivation $v$ is $v=\sum_{j=1}^n a_j \frac{\partial}{\partial x_j}|_{\pi(v)}$ under the local coordinate $x$. This makes $TX$ a manifold and the conditions that $TX \rightarrow X$ is a vector bundle structure is checked easily.

Let's compute the transition function. Suppose on $U \subseteq X$ and $x: U \rightarrow U' \subseteq \RR^n, y: U \rightarrow U'' \subseteq \RR^n$ be two local coordinates. Then for smooth function $f : U \rightarrow \RR$, 
\begin{align*}
    & \frac{\partial }{\partial x_j} f = \frac{\partial (f \circ x^{-1})}{\partial t_j} = \frac{\partial (f \circ y^{-1}) \circ (y \circ x^{-1})}{\partial t_j} = \sum_{k=1}^n \frac{\partial (y \circ x^{-1})_k}{\partial t_j} \frac{\partial (f \circ y^{-1})}{\partial t_k} = \left( \sum_{k=1}^n \frac{\partial (y \circ x^{-1})_k}{\partial t_j} \frac{\partial}{\partial y_k} \right) f
\end{align*}
which gives the relation:
\begin{align*}
    \sum_{j=1}^n a_j \frac{\partial}{\partial x_j} = \sum_{j=1}^n b_j \frac{\partial}{\partial y_j} \text{ where } \begin{bmatrix} \frac{\partial (y\circ x^{-1})_1}{\partial t_1} & \cdots & \frac{\partial (y\circ x^{-1})_1}{\partial t_n} \\
    \vdots & \ddots & \vdots \\ \frac{\partial (y\circ x^{-1})_n}{\partial t_1} & \cdots & \frac{\partial (y\circ x^{-1})_n}{\partial t_n}
    \end{bmatrix} \begin{bmatrix} a_1 \\ \vdots \\ a_n \end{bmatrix} = \begin{bmatrix} b_1 \\ \vdots \\ b_n \end{bmatrix}
\end{align*}

\medskip
\textbf{Example: Cotangent Bundle.}

Cotangent bundle is defined as the dual bundle to the tangent bundle: $T^*X := (TX)^*$. Each of the vector is thus a linear functional on the tangent space. We denote the dual basis to $\{\frac{\partial}{\partial x_1}, \cdots \frac{\partial}{\partial x_n}\}$ by $\{\d x_1, \cdots \d x_n\}$. To compute the transition function, we consider the general situation of map on dual vector space. Suppose $f \in \GL(V)$ for $V \cong \RR^n$ which gives $f^* \in \GL(V^*)$ given by $f^*(\varphi) = (\varphi \circ f)^{-1}$ (we take inverse because taking dual is contravariant). Suppose a basis for $V$ is given by $\{v_1, \cdots v_n\}$ and dual basis for $V^*$ is denoted $\{ v_1^*, \cdots v_n^*\}$. Suppose matrix of $f$ in this basis is $A=(a_{jk})_{jk}$ so that $f(v_j) = \sum_k a_{jk} v_k$.
\begin{align*}
    (f^*(v_j^*))^{-1}(v_k) = (v_j^* \circ f)(v_k)= v_j^* (\sum_{l=1}^n a_{kl} v_l) = a_{kj} \implies (f^*(v_j^*))^{-1} = \sum_{k=1}^n a_{kj} v_k^*
\end{align*}
Thus matrix for $f^*$ is thus given by $(A^T)^{-1}$. Due to this description, we get:
\begin{align*}
    \sum_{j=1}^n a_j \d x_j = \sum_{j=1}^n b_j \d y_j \text{ where }
    \begin{bmatrix} \frac{\partial (y\circ x^{-1})_1}{\partial t_1} & \cdots & \frac{\partial (y\circ x^{-1})_n}{\partial t_1} \\
    \vdots & \ddots & \vdots \\ \frac{\partial (y\circ x^{-1})_1}{\partial t_n} & \cdots & \frac{\partial (y\circ x^{-1})_n}{\partial t_n}
    \end{bmatrix}^{-1} \begin{bmatrix} a_1 \\ \vdots \\ a_n \end{bmatrix} = \begin{bmatrix} b_1 \\ \vdots \\ b_n \end{bmatrix}
\end{align*}

Smooth global sections of $\wedge^p T^*X$ are called differential $p$-forms. The set of differential $p$-forms is denoted by $\mc E^p(X)$. In fact, we will discuss mostly complex-coefficient smooth differential forms $\mc E^p(X)_\CC$, but often the $\CC$ subscript is dropped because we will be discussing the complex-coefficient forms most of the time.

\subsection{Almost Complex Structure and $\dol$-operator}

Suppose we are given a real vector space $V$ of dimension $n$ with an endomorphism $J$ such that $J^2 = -I$. $J$ is called a \emph{complex structure} on $V$. We may extend action of $J$ to the complexification $V_\CC := V \otimes_\RR \CC$ of $V$. This is to be thought as a generalization of multiplication by $i$. Eigenvalues of $J$ are $\pm i$, and using the condition $J^2 = -I$ with Jordan decomposition implies that there is an eigenspace decomposition $V_\CC \cong V^{1,0} \oplus V^{0,1}$ where $V^{1,0}$ is the $(+i)$-eigenspace and $V^{0,1}$ is the $(-i)$-eigensepace. Denote by $\wedge^{p,q}V$ the subspace of exterior algebra $\wedge V$ generated by $p$ elements from $V^{1,0}$ and $q$ elements from $V^{0,1}$. Then we get a bidegree decomposition:
$$\wedge V = \bigoplus_{r=0}^{2n}\bigoplus_{p+q=r} \wedge^{p,q} V$$

This construction has a differential-geometric version:
\begin{definition}
Suppose $X$ is a differentiable manifold of dimension $2n$. If there is a vector bundle isomorphism
$$J: T(X) \rightarrow T(X) $$
such that $J_x^2 = -I_x$ on each $x \in X$, then $J$ is said to be an \emph{almost complex structure} on $X$.
\end{definition}

A complex manifold has an almost complex structure; $J$ is given by multiplication by $i$. 

The decomposition above can be carried out here too, for $V = T^*_x X$. We then get a decomposition of complex-coefficient smooth differential forms into bidegree:
\begin{align*}
    \mc E^r(X)_\CC = \bigoplus_{p+q=r} \mc E^{p,q}(X)
\end{align*}
where $\mc E^{p,q}(X)$ are subbundles generated by $\wedge^{p,q} T^*_x(X)$.

Given a complex manifold with local coordinates $(z_1, \cdots z_n)$ and $z_j = x_j + i y_j$, we make the following definitions:
\begin{align*}
    & \frac{\dol}{\dol z_j} = \frac12 \left( \frac{\dol}{\dol x_j} - i\frac{\dol}{\dol y_j} \right), \frac{\dol}{\dol \bar z_j} = \frac12 \left( \frac{\dol}{\dol x_j} + i\frac{\dol}{\dol y_j} \right) \\
    & \d z_j := \d x_j + i \d y_j, \d \bar z_j := \d x_j - i \d y_j
\end{align*}
These definitions coincide with usual complex differentiation.

Note that $\{\d z_1, \cdots \d z_n\}$ forms a basis for $(T^*X)^{1,0}$ and $\{\d \bar{z}_1, \cdots \d \bar{z}_n\}$ forms a basis for $(T^*X)^{0,1}$.

Let $\pi_{p,q}: \mc E^r (X) \rightarrow \mc E^{p,q}(X)$

We define the operator $\dol, \bar \dol$ by
\begin{align*}
    \dol =& \pi_{p+1,q} \circ \d \\
    \bar\dol =& \pi_{p,q+1} \circ \d
\end{align*}
so that $\d = \dol + \bar \dol$.

If $(z_1, \cdots z_n)$ is a local coordinate, local expressions for $\dol, \bar\dol$ are given by
\begin{align*}
    \dol (\varphi_I \d z_I) &= \sum_{j=1}^n \frac{\dol \varphi_I}{\dol z_j} \d z_j \wedge \d z_I \\
    \bar\dol (\varphi_I \d z_I) &= \sum_{j=1}^n \frac{\dol \varphi_I}{\dol \bar{z}_j} \d \bar{z}_j \wedge \d z_I
\end{align*}

\subsection{Sheaf Theory}

\subsubsection{Basic Properties}

A sheaf is an assignment of data to each open subset of a topological space. The intuition is to think of the assigned data to be the set of all well-behaved functions assigned on that set. However, it can represent more general objects than functions, such as differential forms. In this article, we will be mostly concerned with the sheaf of sections of vector bundles. One advantage of using sheaf to describe vector bundles is that we can use the machinery of sheaf cohomology to derive nontrivial properties of sections of vector bundles.

Given a topological space $X$ and a category $\mc C$, a \emph{presheaf} on $X$ with values in $\mc C$ is a contravariant functor from the category of the open subsets of $X$ to $\mc C$. The category $\mc C$ is usually $\textbf{Set}, \textbf{Ab}, \textbf{Ring}$. A presheaf of (say) ring thus assigns, for each open set $U$, a ring $\mc F(U)$ and for each pair of open sets $U \subseteq V$, a map $\on{res}_{VU} : \mc F(V) \rightarrow \mc F(U)$ so that whenever $U \subseteq V \subseteq W$, $\on{res}_{VU} \circ \on{res}_{WV} = \on{res}_{WU}$. We also write $\on{res}_{VU}(s) = s|_U$, keeping with the intuition of `taking restriction'. An element of $\mc F(U)$ is called \emph{section} of $\mc F$ over $U$.

A \emph{sheaf} on $X$ is a presheaf that satisfies two additional conditions:
\begin{enumerate}
    \item (Identity) Given an open cover $\mc U = \{U_\alpha\}$ of $U \subseteq X$, and sections $s, s' \in \mc F(U)$, if $\forall \alpha, s|_{U_\alpha} = s'|_{U_\alpha}$, then $s=s'$.
    \item (Gluing) Given an open cover $\mc U = \{U_\alpha\}$ of $U \subseteq X$ and sections $s_\alpha \in \mc F(U_\alpha)$ such that $\forall \alpha, \beta, s_\alpha|_{U_{\alpha}\cap U_\beta} = s_\beta|_{U_{\alpha} \cap U_\beta}$, then there is an element $s \in \mc F(U)$ so that $\forall \alpha, s|_{U_\alpha} = s_\alpha$.
\end{enumerate}
For $\mc C = \textbf{Ab}, \textbf{Ring}$, these conditions can be written as exactness of the following chain complex:
\begin{align*}
    0 \rightarrow \mc F(U) \rightarrow \prod_\alpha \mc F(U_\alpha) \rightarrow \prod_{\alpha,\beta} \mc F(U_\alpha \cap U_\beta)
\end{align*}
where maps are given by restrictions.

The \emph{stalk} of a sheaf $\mc F$ at $x \in X$ is defined as:
$$\mc F_{x} := \lim_\rightarrow \mc F(U)$$
where the direct limit is taken over all open neighborhoods $U \supset x$ and the direct system is given by restriction maps. Here, two sections are identified iff they agree locally at $x$, so that this represents infinitesimal data of sections near $x$.

Given two presheaves $\mc F, \mc G$ on $X$ valued in $\mc C$, a \emph{morphism} (or map) between them $\varphi: \mc F \rightarrow \mc G$ is a natural transformation between functors $\mc F$ and $\mc G$. Explicitly, we assign map $\varphi_U: \mc F(U) \rightarrow \mc G(U)$ for each open $U \subseteq X$ which are compatible with restrictions: whenever $U\subseteq V, \varphi_U(s|_U) = (\varphi_V(s))|_U$. A map of presheaves induces map on stalks: $\varphi_x : \mc F_x \rightarrow \mc G_x$ given by $[f] \mapsto [\varphi_U(f)]$.

A map of presheaves $\varphi$ is said to be injective if for each open $U$, $\varphi_U$ is injective. $\varphi$ is said to be surjective if for each $x \in X$, the stalk-induced map $\varphi_x$ is surjective. $\varphi$ is isomorphism if it is invertible. Additionally, the following are true:
\begin{proposition}
    If $\varphi$ is a map of sheaves,
    \begin{enumerate}
        \item $\varphi$ is injective iff $\forall x, \varphi_x$ is injective.
        \item $\varphi$ is isomoprhism iff $\forall x, \varphi_x$ is isomorphism.
        \item $\varphi$ is isomorphism iff it is injective and surjective.
    \end{enumerate}
\end{proposition}
\begin{proof}
    The first statement follows from the identity axiom of sheaves. For the second statement, the forward implication is straightforward and the converse is proven by gluing local preimages given by stalk surjectivity. The third statement follows from the first two.
\end{proof}

Given a presheaf $\mc F$, we can define its \emph{sheafification} as assignment of `compatible germs'. More precisely, we define:
\begin{align*}
    \mc F^\dagger(U) = \{\sigma: U \rightarrow \bigcup_x \mc F_x | \sigma(x) \in \mc F_x, \forall x \in U, \exists V, s \in \mc F(V): x \in V \subseteq U, \forall y \in V, s_y = \sigma(y) \}
\end{align*}
Note that an equivalent definition is to define a section as a collection of sections over an open cover that agree on overlaps, modulo having the same germs.

Sheafification can also be characterized using the following universal property: whenever there is a presheaf morphism $\mc F \rightarrow \mc G$ where $\mc G$ is a sheaf, this morphism factors through $\mc F \rightarrow \mc F^\dagger$. The fact that the sheafification satisfies this universal property is verified by gluing images of the compatible sections.

Given a morphism $\varphi: \mc F \rightarrow \mc G$ of sheaves, we define $\ker \varphi$ to be the assignment $U \mapsto \ker (\varphi(U))$ (with restriction well-defined due to $\varphi$ being a natural transformation) and $\on{im} \varphi$ to be the sheafification of the assignment $U \mapsto \on{im}(\varphi(U))$. We say that a sequence of sheaf morphisms $\mc F \xrightarrow{\alpha} \mc G \xrightarrow{\beta} H$ is exact at $\mc G$ if $\ker \beta = \on{im} \alpha$.

Given a sheaf of rings $\mc O_X$, we say that a sheaf of abelian groups $\mc F$ is a \emph{sheaf of $\mc O_X$-module} if $\forall U, \mc F(U)$ is a $\mc O_X(U)$-module and module action is compatible with restriction: $V \subseteq U, a \in \mc O_X(U), s \in \mc F(U) \implies (a \cdot s)|_V = a|_V \cdot s|_V$.

Given two sheaves of $\mc O_X$-modules $\mc F, \mc G$, let $U \mapsto \mc F(U) \otimes_{\mc O_X(U)} \mc G(U)$ define a presheaf denoted as $\mc F \otimes'_{\mc O_X} \mc G$. This is not generally a sheaf yet, so we define the \emph{tensor product} $\mc F \otimes_{\mc O_X} \mc G$ of sheaves $\mc F, \mc G$ to be the sheafification of $\mc F \otimes'_{\mc O_X} \mc G$.

Given a smooth vector bundle $E \rightarrow X$, we can define the sheaf of smooth sections of $E$, denoted $\mc E(E)$, to be the assignment 
$$\mc E: U \mapsto \{\text{smooth sections of $E$ over $U$}\}$$
with the literal restriction maps. Similarly we can define the sheaf of holomorphic sections of a holomoprhic vector bundle, which is denoted $\mc O(E)$.

A meromorphic function on a complex manifold is defined as a compatible collection of formal quotient of holomorphic functions: 
$$\mc M(U) := \{ \frac{f_\alpha}{g_\alpha} | f_\alpha, g_\alpha \in \mathcal O(U_\alpha), f_\alpha g_\beta = f_\beta g_\alpha \text{ on $U_\alpha \cap U_\beta$} \}_\alpha$$

\subsubsection{Sheaf cohomology}

In this section we assume that the topological space $X$ over which sheaves will be defined is paracompact.

We now define sheaf cohomology, a powerful tool that describes obstructions to lifting. When given a short exact sequence of sheaves, the corresopnding sequence of global sections loses right-exactness:
\begin{align*}
    & 0 \rightarrow \mc F \rightarrow \mc G \rightarrow \mc H \rightarrow 0 \text{ is exact} \\
    \implies & 0 \rightarrow \mc F(X) \rightarrow \mc G(X) \rightarrow \mc H(X) \text{ is exact}
\end{align*}
Indeed, as seen from the definition of surjective sheaf morphism, surjectivity of sheaf morphism $\mc G \rightarrow \mc H$ doesn't imply surjectivity of section-level morphisms like $\mc G(X) \rightarrow \mc H(X)$. Sheaf cohomology groups fill up this void:
\begin{align*}
    & 0 \rightarrow \mc F(X) \rightarrow \mc G(X) \rightarrow \mc H(X) \rightarrow H^1(X, \mc F) \rightarrow H^1(X, \mc G) \rightarrow H^1(X, \mc H) \rightarrow H^2(X, \mc F) \rightarrow \cdots \text{ is exact}
\end{align*}
and thus when we have $H^1(X, \mc F)=0$, the map $\mc G(X) \rightarrow \mc H(X)$ is always surjective.

We define sheaf cohomology using canonical resolution. Given a sheaf $\mc S$, define $\mc C^0(U)= \{f: U \rightarrow \bigcup_{x \in U} \mc S_x | \pi \circ f = \on{id}_U \}$. Then we get an injection $0 \rightarrow \mc S \rightarrow \mc C^0$, which extends to
$$0 \rightarrow \mc S \rightarrow \mc C^0 \rightarrow \mc F^1 \rightarrow 0$$
where $\mc F^1$ is defined as the quotient sheaf $\mc F^1 := \mc C^0 / \mc S$. Define also $\mc C^1 = \mc C^0(\mc F^1)$. Letting $\mc F^1$ take the place of $\mc S$, we continue the construction. Given $\mc F^i, \mc C^i$, define $\mc F^{i+1} = \mc C^i / \mc F^i, \mc C^{i+1} = \mc C^0(\mc C^i / \mc F^i)$. Then we get exact sequences:
\begin{align*}
    0 \rightarrow \mc S \rightarrow \mc C^0 \rightarrow \mc F^1 \rightarrow 0 \\
    0 \rightarrow \mc F^1 \rightarrow \mc C^1 \rightarrow \mc F^2 \rightarrow 0 \\
    0 \rightarrow \mc F^2 \rightarrow \mc C^2 \rightarrow \mc F^3 \rightarrow 0 \\
    \vdots
\end{align*}
and we can compose maps $\mc C^i \rightarrow \mc F^{i+1} \rightarrow \mc C^{i+1}$ to splice these short exact sequences into one long exact sequence:
\begin{align*}
    0 \rightarrow \mc S \rightarrow \mc C^0 \rightarrow \mc C^1 \rightarrow \mc C^2 \rightarrow \cdots
\end{align*}
and this is called the canonical resolution of sheaf $\mc S$. Cohomology of the complex $0 \rightarrow \mc C^*(X)$ is called the sheaf cohomology groups of $\mc S$. We denote the sheaf cohomology group at $\mc C^p$ by $H^p(X, \mc S)$.

This construction is rather abstract, but it turns out that we can use other resolutions to compute sheaf cohomology groups:

\begin{theorem}
    For a sheaf $\mc F$ over space $X$, suppose
    $$0 \rightarrow \mc F \rightarrow \mc A^0 \rightarrow \mc A^1 \rightarrow \cdots $$
    is an acyclic resolution; $\forall p>0, q\ge 0, H^p(X, \mc A^q)=0$. Then cohomology of the complex
    $$0 \rightarrow \mc A^0(X) \rightarrow \mc A^1(X) \rightarrow \cdots$$
    at $\mc A^p(X)$ is isomorphic to $H^p(X, \mc F)$.
\end{theorem}

The acyclicity requirement seems rather circular, but it turns out that there is a family of sheaves for which it is easy to show the sheaf cohomology groups to be vanishing.

If $\mc F$ is a sheaf over $X$, $\mc F$ is said to be \emph{soft} if for any closed set $K \subseteq X$, the map $\mc F(X) \rightarrow \mc F(K)$ is surjective, where $\mc F(K)$ is defined as the collection of germ-sections $K \rightarrow \bigcup_{x \in K} \mc F_x$ that are compatible at intersections. 

It can be shown that:
\begin{proposition}
Suppose $\mc F, \mc G, \mc H$ are sheaves over $X$ and $\mc F$ is soft. Then whenever
$$0 \rightarrow \mc F \rightarrow \mc G \rightarrow \mc H \rightarrow 0$$
is exact, the sequence
$$ 0 \rightarrow \mc F(X) \rightarrow \mc G(X) \rightarrow \mc H(X) \rightarrow 0 $$
is also exact.
\end{proposition}

Corollaries include:
\begin{corollary}
    Suppose that following sequence is exact:
    $$0 \rightarrow \mc F \rightarrow \mc G \rightarrow \mc H \rightarrow 0$$
    If $\mc F, \mc G$ are soft, then $\mc H$ is also soft.
\end{corollary}
\begin{proof}
    For a closed set $K \subseteq X$, the sequence 
    $$0 \rightarrow \mc F(K) \rightarrow \mc G(K) \rightarrow \mc H(K) \rightarrow 0$$
    is exact since $\mc F$ is soft. Then a section $s \in \mc H(K)$ lifts to a section in $\mc G(K)$, which in turn lifts to a section in $\mc G(X)$ by softness of $\mc G$, and image of this section in $\mc H(X)$ is the desired lift of $s$.
\end{proof}

\begin{corollary}
    Given a sequence of soft sheaves
    $$0 \rightarrow \mc F_0 \rightarrow \mc F_1 \rightarrow \cdots $$
    the sequence of global sections is also exact:
    $$0 \rightarrow \mc F_0(X) \rightarrow \mc F_1(X) \rightarrow \cdots $$
\end{corollary}
\begin{proof}
    Let $\mc K_i = \on{ker}(\mc F_i \rightarrow \mc F_{i+1})$. Then we get short exact sequences
    $$0 \rightarrow \mc K_i \rightarrow \mc F_i \rightarrow \mc K_{i+1} \rightarrow 0$$
    Now we can induct using the previous corollary; $\mc K_1 = \mc F_0$ is soft and thus $\mc K_1$ is soft, and inductively we see that all $\mc K_i$ are soft, and thus for all $i$,
    $$0 \rightarrow \mc K_i(X) \rightarrow \mc F_i(X) \rightarrow \mc K_{i+1}(X) \rightarrow 0 $$
    is exact. Splicing these sequences, the claim is proven.
\end{proof}

\begin{corollary}
    A soft sheaf $\mc S$ is acyclic.
\end{corollary}
\begin{proof}
    The canonical resolution is soft because each $\mc C^i$ is soft and we can inductively see that all quotients $\mc F^i = \mc C^{i-1}/\mc F^{i-1}$ are soft too. Thus we can use the previous corollary to see that all cohomology groups of the complex of global sections vanish.
\end{proof}

\begin{proposition}
    If $\mc M$ is a $\mc R$-module for soft sheaf of rings $\mc R$, then $\mc M$ is also soft.
\end{proposition}
\begin{proof}
    For a closed set $K$, consider $s \in \mc M(K)$. Its definition extends to an open neighborhood of $K$, say $U$. Now consider a section $\rho \in \mc R(K \cup U^c)$ which is defined to be 1 at $K$ and 0 at $U^c$. Due to softness of $\mc R$, $\rho$ extends to all of $X$ and $\rho \cdot s$ defines an extension of $s$ to $X$.
\end{proof}

Examples of soft sheaves include the sheaf of differentiable singular $p$-cochains $\mc S_\infty^p(X)$; $\mc S_\infty^0(X) = \mc C^0(X,\RR)$ is a soft sheaf and $\mc S_\infty^p(X)$ is a sheaf of $\mc S_\infty^0(X)$-modules.

We also define another type of sheaves called \emph{fine} sheaves which turn out to be soft too. These are sheaves with \emph{partition of unity}; for every locally finite open cover $\{U_i \}_{i \in I}$, we have $\{ \eta_i: \mc F \rightarrow \mc F \}_{i \in I}$ that satisfy:
\begin{enumerate}
    \item $\sum_i \eta_i \equiv 1$; sum well-defined due to local finiteness
    \item Stalk maps $(\eta_i)_x \equiv 0$ at a neighborhood of complement of $U_i$.
\end{enumerate}

\begin{proposition}
    Over a paracompact space $X$, fine sheaves are soft.
\end{proposition}
\begin{proof}
    Let $\mc F$ be a fine sheaf. Suppose $K \subseteq X$ is closed and take a section $s \in \mc F(U)$ which is defined as a germ-section. Then we have open cover $\{U_i\}_{i \in I}$ of $K$ with sections $s_i \in \mc F(U_i)$ such that $(s_i)_x = s(x)$. Let $U_0 = K^c$ and refine the cover $\{U_0 \} \cup \{U_i\}_{i \in I}$ of $X$ into a locally finite cover using paracompactness. We can now apply fineness axiom to find partition of unity $\{ \eta_i \}$ subject to $\{U_0 \} \cup \{U_i\}_{i \in I}$ and find $\eta(U_i)(s_i) \in \mc F(U_i)$, which is zero at a neighborhood of $U_i^c$ and thus we can glue it with zero section to find $\tilde s_i$. Now by adding all $\tilde s_i$, we get a lift of $s$.
\end{proof}

Examples of fine sheaves include $\mc E^r$ and $\mc E^{p,q}$ over (almost complex) manifold $X$ where scalar multiplication by the usual partition of unity defines the partition of unity associated with fine sheaves. These are therefore soft sheaves too.

From above discussions, we obtain following two important theorems:

\begin{theorem}[de Rham]
    For a differentiable manifold $X$, singular cohomology is isomorphic to the de Rham cohomology.
\end{theorem}
\begin{proof}
    We have two resolutions:
    \begin{align*}
        & 0 \rightarrow \RR \rightarrow \mc E^* \\
        & 0 \rightarrow \RR \rightarrow \mc S_\infty^*
    \end{align*}
    which are associated by a map of chain complexes:
    \begin{align*}
        & \mc E^*(U) \rightarrow \mc S_\infty^p(U) \\
        & \omega \mapsto (\sigma \mapsto \int_\sigma \omega)
    \end{align*}
    whose commutativity as chain map is ensured by Stokes' theorem. The sheaves $\mc E^*$ are fine, thus soft, and thus acyclic, and sheaves $\mc S_\infty^*$ are soft, and thus acyclic. Thus the cohomology defined by both of these resolutions are isomorphic to the sheaf cohomology of the constant sheaf $\RR$.
\end{proof}

\begin{theorem}[Dolbeault]
    For a complex manifold $X$, cohomology of the sheaves of differential forms can be computed as follows:
    \begin{align*}
        H^q(X,\Omega^p) \cong H^q( \mc E^{p,*}(X) )
    \end{align*}
\end{theorem}
\begin{proof}
    The resolution
    $$0 \rightarrow \Omega^p \rightarrow \mc E^{p,*}$$
    is soft and thus acyclic.
\end{proof}

Using the fact that tensoring by locally free sheaf is an exact functor, we also see the following generalization:
\begin{theorem}
    For a complex manifold $X$ and a holomorphic vector bundle $E$ over it, cohomology of the sheaves of differential forms with coefficients in $E$ can be computed as follows:
    \begin{align*}
        H^q(X,\Omega^p(E)) \cong H^q( \mc E^{p,*}(X,E) )
    \end{align*}
\end{theorem}

\subsubsection{\v Cech cohomology}

\emph{\v Cech cohomology} groups can also be defined, and they are equal to the sheaf cohomology groups when the underlying space is paracompact.

Given an open cover $\mc U = \{ U_\alpha \}$, let a \emph{$k$-simplex} be an intersection $U_{\alpha_1} \cap \cdots cap U_{\alpha_k}$. This terminology makes sense if we `draw' a simplicial complex over a topological space where we draw a vertex for each open set, a line for each nonempty intersection, a triangle for each nonempty triple-intersection, and so on. Such a simplicial complex is called the \emph{nerve} of the cover, and the nerve theorem tells us that when all intersections of the cover are contractible and the underlying space is paracompact, homotopy type of the nerve is the same as the original space. 

We let the \emph{$k$-cochain group} $C^k(\mc U, \mc F)$ to be a map that assigns to each $k$-cochain $\sigma = U_{\alpha_1} \cap \cdots \cap U_{\alpha_k}$ an element of $\mc F(\sigma)$. 

Define \emph{boundary} of a $k$-simplex $\sigma = U_{\alpha_1} \cap \cdots \cap U_{\alpha_k}$ by:
\begin{align*}
    \partial \sigma := \sum_{j=1}^k (-1)^j \sigma_j \text{ where } \sigma_j := \bigcap_{l \neq j} U_{\alpha_l}
\end{align*}
and also define the \emph{coboundary} of a $k$-cochain to be the following $(k+1)$-cochain:
\begin{align*}
    \delta \varphi : \sigma \mapsto \sum_{j=1}^k (-1)^j \varphi(\sigma_j)
\end{align*}
The \emph{\v Cech cohomology} of $\mc F$ for $\mc U$ is defined as the cohomology of the chain complex $C^\bullet(\mc U, \mc F)$, and is denoted $H^\bullet(\mc U, \mc F)$. The \v Cech cohomology of $\mc F$ is defined as the direct limit of the cohomology over open covers:
$$H^k(X, \mc F) := \lim_\rightarrow H^k(\mc U, \mc F)$$
where the direct system is given by refinement-induced maps on cohomology.

\subsection{Metric, Connection, Curvature and Chern Class}

In this section, we will be concerned with various constructions on a complex vector bundle.

\begin{definition}
    A \emph{frame} on a complex vector bundle $E\rightarrow X$ at $x \in X$ is a finite set of sections $\{ e_1, \cdots e_r\} \subset \mc E(U,E)$ such that $\forall y \in U, \{ e_1(y), \cdots e_r(y)\}$ is a basis for $E_y$.
\end{definition}

If $A : U \rightarrow \GL(\CC^k)$ is smooth, we can define change of frame of $f$ induced by $A$:
\begin{align*}
    f \mapsto fA, \begin{bmatrix} e_1 & \cdots & e_r \end{bmatrix} \mapsto \begin{bmatrix} e_1 & \cdots & e_r \end{bmatrix} \begin{bmatrix} g_{11} & \cdots & g_{1r} \\ \vdots & \ddots & \vdots \\ g_{r1} & \cdots & g_{rr} \end{bmatrix}
\end{align*}

Existence of frame over an open set is equivalent to the vector bundle being trivial there; a local trivialization immediately gives a frame by pullback and whenever we have a frame $\{e_j\}$ over $U$, we can define $U \times \CC^k \rightarrow \pi^{-1}(U)$ by $(x,v) \mapsto (e_1(x) v_1 + \cdots e_k(x) v_k)$. For a section $\xi \in \mc E(U,E)$ and frame $f$ of $E$ over $U$, we can write
\begin{align*}
    \xi(f) := \begin{bmatrix} \xi^1(f) \\ \vdots \\ \xi^r(f) \end{bmatrix}, f \cdot \xi(f) = \begin{bmatrix} e_1 & \cdots & e_r \end{bmatrix} \begin{bmatrix} \xi^1(f) \\ \vdots \\ \xi^r(f) \end{bmatrix} = \xi
\end{align*}

As above, we will frequently think of $f$ as a row vector $f = \begin{bmatrix} e_1 & \cdots & e_r \end{bmatrix}$.

We have the following change of frame formula:
\begin{proposition}
For $\xi \in \mc E(V,E)$, a frame $f$ of $E$ over $V \subseteq U$ and a change of frame $A: U \rightarrow \GL(\CC^k)$,
$$\xi(fA) = A^{-1} \xi(f)$$
\end{proposition}
\begin{proof}
    \begin{align*}
        fA \cdot \xi(fA) = f \cdot \xi(f) \iff \xi(fA) = A^{-1} \xi(f)
    \end{align*}
\end{proof}

We can develop analogous concepts of holomorphic vector bundles and holomorphic change of frames.

Tensor product of bundles and sheaves correspond to each other as following:
\begin{proposition}
    Suppose $E,F$ are smooth vector bundles. Then we have a sheaf isomorphism:
    \begin{align*}
        \mc E(E) \otimes_{\mc E} \mc E(F) \cong \mc E(E \otimes F)
    \end{align*}
\end{proposition}
\begin{proof}
    Using frames of $E,F$, we can give a local expression of $\mc E(E \otimes F)$ as the image of the natural map from $\mc E(U,E) \otimes_{\mc E(U)} \mc E(U,F)$. This gives a presheaf morphism from the pre-sheafified tensor product presheaf to $\mc E(E \otimes F)$, and it factors through $\mc E(E) \otimes_{\mc E} \mc E(F)$ by the universal property of sheafification. Surjectivity was just shown, and injectivity too is easy to verify.
\end{proof}

Given a complex vector bundle $E \rightarrow X$, we define the sheaf $\mc E^p(E) := \mc E((\wedge^p T^*X)_{\CC} \otimes E)$. A section of $\mc E^p(E)$ over $U$ can locally be written in the form
$$\xi = \sum_{k=1}^r \xi^k(f) \otimes e_k$$
where $f=\{ e_j \}$ is a frame and $\xi^k(f)$ is a $p$-form. This representation comes from locally writing $\xi$ as
$$\sum_{l,k} \varphi_{l,k} \cdot \omega_l(x) \otimes e_k(x) = \sum_k \left( \sum_l \varphi_{l,k} \omega_l(x) \right) \otimes e_k(x) = \sum_k \xi^k(f) \otimes e_k(x)$$
where $\{ \omega_j\}$ is a frame of differential $p$-forms near $x$. Note that $\xi^k(f)$ is well-defined any frame $f$ of $E$ over $U$, regardless local trivialization for differential forms.

Writing $\xi(f) = \begin{bmatrix} \xi^1(f) & \cdots & \xi^r(f) \end{bmatrix}^T$, we also see that the change of frame formula from before carries over:
$$\xi(fA) = A^{-1} \cdot \xi(f)$$

\begin{definition}
    For a complex vector bundle $E \rightarrow X$, define a \emph{Hermitian metric} $h$ on it to be assignment of Hermitian inner product on each fiber $E_x$ such that for each smooth section $\xi, \eta \in \mc E(U,E)$, the map $x \mapsto \langle \xi(x), \eta(x) \rangle$ is smooth.
\end{definition}

Equivalently, we can require that for any frame $\{e_j\}$ over $U$, the map $h(f)_{j,k} = x \mapsto \langle e_k(x),e_j(x) \rangle$ is smooth. Note that $\langle \xi, \eta \rangle = \sum_{j,k} \overline{\eta^j(f)} \xi^k h_{jk}(f) = \eta(f)^* h(f) \xi(f)$ (where for $M$, $M^*$ is its conjugate transpose). We also get a change of frame formula:
\begin{proposition}
    For a Hermitian metric $h$ of complex line bundle $E$ over $X$, frame $f$ of $E$ over some $U \subseteq X$ and a change of frame $A: U \rightarrow \GL(\CC^k)$,
    \begin{align*}
         h(fA) = A^* h(f) A
    \end{align*}
\end{proposition}
\begin{proof}
    \begin{align*}
        & \forall \xi, \eta, \eta(f)^* h(f) \xi(f) = \langle \xi, \eta \rangle = \eta(Af)^* h(Af) \xi(Af) = \eta(f) (A^*)^{-1} h(fA) A^{-1} \xi(f) \\
        \implies & h(fA) = A^* h(f) A
    \end{align*}
\end{proof}

\begin{definition}
For a complex vector bundle $E \rightarrow X$, define a \emph{connection} on it to be a $\CC$-linear map:
    $$D : \mc E(X, E) \rightarrow \mc E^1(X, E)$$
    such that
    $$D(\varphi \xi) = \d \varphi \cdot \xi + \varphi \cdot (D \xi)$$
    where $\varphi \in \mc E(X)$ and $\xi \in \mc E(X, E)$.
\end{definition}

Local description of connection can be given using a matrix of 1-forms. Suppose $f=\{e_j\}$ is a frame of $E$ over $U$. Let
$$D(e_k) = \sum_{j=1}^r \theta_{j,k}(D,f) e_j$$
Then we may write $\theta(D,f) = (\theta_{j,k}(D,f))_{j,k}$, a matrix of 1-forms. 
\begin{proposition}
    For a connection $D$ and frame $f$, the following holds:
    $(D \xi)(f) = (\d + \theta)\xi(f)$
\end{proposition}
\begin{proof}
    \begin{align*}
        D\xi &= \sum_{k=1}^r D(\xi^k(f) e_k) \\
        &= \sum_{k=1}^r \d \xi^k(f) e_k + \xi^k(f) D(e_k) \\
        &= \sum_{k=1}^r \left( \d \xi^k(f) e_k + \xi^k(f) \sum_{j=1}^r \theta_{jk}e_j \right) \\
        &= \sum_{k=1}^r \left( \d \xi^k(f) + \sum_{j=1}^r \xi^j(f) \theta_{kj} \right) e_k
    \end{align*}
\end{proof}
Here and onwards, we will often drop the wedge product for multiplication of forms. 

We have the following change of frame formula for local expression of connection:
\begin{proposition}
    For a connection $D$, a frame $f$ and a change of frame $A$, we have
    $$\theta(fA) = A^{-1} \theta(f) A + A^{-1} \d A$$
\end{proposition}
\begin{proof}
    \begin{align*}
        (D\xi)(fA) &= (\d + \theta(fA))\xi(fA) \\
        &= \d A^{-1} \xi(f) + A^{-1} \d \xi(f) + \theta(fA) A^{-1} \xi(f)
    \end{align*}
    So that
    \begin{align*}
        (fA)((D\xi)(fA)) &= fA\d A^{-1} \xi(f) + f \d \xi(f) + f A \theta(fA) A^{-1} \xi (f) \\
        = f (D\xi)(f) &= f \d \xi(f) + f \theta(f) \xi(f) \\
        \implies & A d A^{-1} + A \theta(fA) A^{-1} = \theta (f) \\
        \implies & \theta(fA) = A^{-1} \theta(f) A - (\d A^{-1}) A = A^{-1} \theta(f) A + A^{-1} \d A
    \end{align*}
    where we used $0 = \d I = \d (A^{-1} A) = A^{-1} \d A + (\d A^{-1}) A$ at the end.
\end{proof}

Associated to each connection $D$ is \emph{curvature} $\Theta$. We will define it using compatible local expressions, as an element of $\mc E^2(X, \on{End}(E))$. To do this, let's consider how an element $\chi \in \mc E^2(X, \on{End}(E))$ is described, and how it behaves under a change of coordinate. Suppose $f$ is a frame of $E$ over $U$. Then as we saw before in discussing differential forms with coefficients in vector bundle, the following map is an isomorphism:
$$\mc E^2(U) \otimes_{\mc E(U)} \mc E(U, \on{End}(E)) \rightarrow \mc E^2(U, \on{End}(E))$$
Due to frame $f$, an element of $\mc E(U, \on{End}(E))$ can be thought of as a matrix of smooth functions. Thus, tensoring with $\mc E^2(U)$, we can think of an element of $\mc E^2(U, \on{End}(E))$ as a matrix of 2-forms. This induces a map $\mc E(X,E) \rightarrow \mc E^2(X,E)$. Then we get the following change of frame formula:
\begin{align*}
    & (\chi \xi)(f) = \chi(f) \xi(f) \\
    & (\chi \xi)(fA) = \chi(fA) \xi(fA) = \chi (fA) A^{-1} \xi (f) \\
    \implies & f \chi(f) \xi(f) = f (\chi \xi)(f) = (fA)(\chi \xi)(fA) = fA \chi (fA) A^{-1} \xi (f) \\
    \implies & \chi(fA) = A^{-1} \chi(f) A
\end{align*}
Conversely a collection of matrices of 2-forms associated to frames satisfying this change of frame formula gives a well-defined element of $\mc E^2(X,\on{End}(E))$.

Now we go back to $\Theta$; let $\Theta(D,f) = \d \theta(f) + \theta(f)^2$ be a matrix of 2-forms. Then
\begin{align*}
    \Theta(D,fA) =& \d \theta(fA) + \theta(fA)^2 \\
    =& \d (A^{-1} \theta A + A^{-1} \d A) + (A^{-1} \theta A + A^{-1} \d A)^2 \\
    =& (\d A^{-1}) \theta A + A^{-1} (\d \theta) A - A^{-1} \theta (\d A) + (\d A^{-1})(\d A) \\
    & + A^{-1} \theta^2 A + A^{-1} (\d A) A^{-1} (\d A) + A^{-1} \theta (\d A) + A^{-1} (\d A) A^{-1} \theta A \\
    =& A^{-1} \theta^2 A + A^{-1} (\d \theta) A \\
    =& A^{-1} \Theta(D, f) A
\end{align*}
by using $\d A^{-1} A = - A^{-1} \d A$ and $\d(\omega \wedge \eta) = \d \omega \wedge \eta + (-1)^p \omega \wedge \d \eta$ when $\omega$ is a $p$-form.

Thus $\Theta$ transforms correctly under a change of frame and it defines a an element of $\mc E^2(X,\on{End}(E))$. We thus can define:

\begin{definition}
    Given a connection $D$, curvature $\Theta$ is defined as an element of $\mc E^2(X,\on{End}(E))$ which defines a mapping $\mc E(X,E) \rightarrow \mc E^2(X, E)$ locally given by
    $$\Theta(f) = \d \theta(f) + \theta(f)^2$$
\end{definition}

We have some interesting calculations regarding $\Theta$.
\begin{proposition}
    Viewing $\Theta$ as a mapping $\mc E(X,E) \rightarrow \mc E^2(X,E)$, we locally have:
    $$\Theta = (\d + \theta)^2$$
\end{proposition}
\begin{proof}
    \begin{align*}
        & (\d + \theta)(\d + \theta)\xi(f) \\
        =& \d^2\xi(f) + \theta \d \xi(f) + \d (\theta \xi(f)) + \theta^2 \xi(f) \\
        =& \theta \d \xi(f) + (\d \theta)\xi(f) - \theta \d \xi(f) + \theta^2 \xi(f) \\
        =& (\d \theta + \theta^2) \xi(f) \\
        =& \Theta(f) \xi(f)
    \end{align*}
\end{proof}

\begin{proposition}[Bianchi]
    Viewing $\Theta$ as a mapping $\mc E(X,E) \rightarrow \mc E^2(X,E)$, we locally have:
    $$\d \Theta = [\Theta, \theta]$$
\end{proposition}
\begin{proof}
    \begin{align*}
         \d \Theta &= \d (\d \theta + \theta^2 ) = \d^2 \theta + (\d\theta) \theta - \theta (\d \theta)  + \theta^2 \d = (\d \theta) \theta - \theta (\d \theta) \\
        [\Theta, \theta] &= (\d \theta + \theta^2) \theta - \theta (\d \theta + \theta^2) = (\d \theta) \theta - \theta (\d \theta)
    \end{align*}
\end{proof}

Using change of frame formula, we may extend the definition of $D$ as a map $\mc E^p(X,E) \rightarrow \mc E^{p+1}(X,E)$.
\begin{proposition}
    By letting $D = \d + \theta$ act on forms, we may extend the definition of $D$ and $\Theta$:
    \begin{align*}
        D : \mc E^p (X,E) \rightarrow \mc E^{p+1} (X,E) \\
        \Theta : \mc E^p (X,E) \rightarrow \mc E^{p+2} (X,E)
    \end{align*}
\end{proposition}
\begin{proof}
    It suffices to show that given a vector of $p$-forms $\xi(f)$ transforms correctly under change of frame formula:
    \begin{align*}
        \d \xi(fA) + \theta(fA) \xi(fA) =& \d (A^{-1} \xi(f)) + A^{-1} \theta(f) \xi(f) + A^{-1} \d A \xi (fA) \\
        =& (\d A^{-1}) \xi(f) + A^{-1} (\d \xi(f)) + A^{-1} \theta \xi(f) + A^{-1} (\d A) A^{-1} \xi(f) \\
        =& A^{-1} \d \xi(f) + A^{-1} \theta \xi(f) \\
        =& A^{-1} (\d + \theta) \xi(f)
    \end{align*}
    and therefore $D \xi$ is well-defined. Also, we can extend $\Theta$ since it is locally given as $D^2$.
\end{proof}

We now describe connections that are compatible with metric.
\begin{definition}
    Given a complex vector bundle $E$ with Hermitian metric $h$, a connection $D$ is said to be \emph{compatible} with $h$ if
    $$\d \langle \xi, \eta \rangle = \langle D \xi, \eta \rangle + \langle \xi, D \eta \rangle$$
    for every section $\xi, \eta$.
\end{definition}

Metric compatibility has a simple criterion:
\begin{proposition}
    A connection $D$ is compatible with metric $h$ iff for every frame $f$,
    $$\d h = h\theta + \theta^* h$$
    where $h=h(f), \theta=\theta(f)$ here.
\end{proposition}
\begin{proof}
    Let $f=\{e_j\}$. Then
    \begin{align*}
        \d h_{jk} =& \d \langle e_j, e_k \rangle \\
        =& \langle D e_j , e_k \rangle + \langle e_j, D e_k \rangle \\
        =& \langle \sum_{\alpha=1}^r \theta_{\alpha j} e_\alpha, e_k \rangle + \langle e_j, \sum_{\alpha=1}^r \theta_{\alpha k} e_\alpha \rangle \\
        =& \sum_{\alpha=1}^r \bar{\theta}_{\alpha j} h_{\alpha k} + \sum_{\alpha=1}^r \theta_{\alpha k} h_{j\alpha} \\
        =& (h \theta + \theta^* h)_{jk}
    \end{align*}
    and conversely if this holds,
    \begin{align*}
        \d \langle \xi, \eta \rangle =& \d \langle \sum_{j=1}^r \xi^j e_j, \sum_{k=1}^r \eta^k e_k \rangle\\
        =& \d \sum_{j,k=1}^r \xi^j \eta^k h_{jk} \\
        =& \sum_{j,k=1}^r ((\d \xi^j) \eta^k + \xi^j (\d \eta^k)) h_{jk} + \xi^j \eta^k \d h_{jk} \\
        =& (\d \eta)^* h \xi + \eta^* (\d h) \xi + \eta^* h (\d \xi) \\
        =& (\d \eta)^* h \xi + \eta^* h \theta \xi + \eta^* \theta^* h \xi + \eta^* h (\d \xi) \\
        =& (\d \eta + \theta \eta)^* h \xi + \eta^* h (\d \xi + \theta \xi \\
        =& \langle \xi, D \eta \rangle + \langle D \xi, \eta \rangle
    \end{align*}
\end{proof}

For a connection on a complex manifold, we may use the complex structure to give a decomposition:
\begin{align*}
    & D = D' + D'' \\
    & D: \mc E(E) \rightarrow \mc E^1(E) \cong \mc E^{1,0}(E) \oplus \mc E^{0,1}(E) \\
    & D' : \mc E(E) \rightarrow \mc E^{1,0}(E) \\
    & D'' : \mc E(E) \rightarrow \mc E^{0,1}(E)
\end{align*}
Note that due to type consideration, we get
\begin{align*}
    D' =& \dol + \theta^{1,0} \\
    D'' =& \bar\dol + \theta^{0,1}
\end{align*}

\begin{proposition}
    For a holomorphic vector bundle $E \rightarrow X$ with Hermitian metric $h$, $h$ induces a \emph{canonical connection} $D$ such that $D$ is compatible with metric and for each holomorphic section $\xi \in \mc O(U, E)$, $D'' \xi = 0$.
\end{proposition}
\begin{proof}
    For holomorphic section $\xi$ and a frame $f$ for $E$, $D''\xi(f) = \bar\dol \xi(f) + \theta^{(0,1)} \xi(f) = \theta^{(0,1)} \xi(f)$ and thus the condition for $D'' \xi = 0$ is equivalent to $\theta$ having the type $(1,0)$.
    
    Suppose there is a connection $D$ satisfying the conditions. Then metric compatibility gives $\d h = h \theta + \theta^* h$, and by decomposing into type $(1,0)$ and $(0,1)$ we get
    \begin{align*}
        & \dol h = h \theta \\
        & \bar\dol h = \theta^* h
    \end{align*}
    By conjugation, the second equation is equivalent to the first and thus it suffices to consider the first. Now we prove that the following definition:
    $$\theta := h^{-1} \dol h$$
    gives a well-defined connection. For a holomorphic change of frame $A$,
    \begin{align*}
        \theta(fA) &= h^{-1}(fA) \dol h(fA) \\
        &= A^{-1} h^{-1} (A^*)^{-1} \dol(A^* h A) \\
        &= A^{-1} h^{-1} (A^*)^{-1} ((\dol A^*) h A + A^* (\dol h) A + A^* h (\dol A)) \\
        &= A^{-1} h^{-1} (A^*)^{-1} (0 + A^* (\dol h) A + A^* h (\d A)) \\
        &= A^{-1} h^{-1} (\dol h )A + A^{-1} \d A \\
        &= A^{-1}\theta A + A^{-1} \d A
    \end{align*}
    where $h=h(f)$ in the calculation to avoid clutter. Thus $\theta$ as defined transforms correctly under change of frame and gives a well-defined connection.
\end{proof}

\begin{proposition}
    For a holomorphic line bundle $E \rightarrow X$ with Hermitian metric $h$ and canonical connection $D$, the following hold:
    \begin{align*}
        & \dol \theta = - \theta^2 \\
        & \Theta = \bar\dol \theta \text{ (locally)} \\
        & \dol \Theta = 0, \dol \Theta = [\Theta, \theta]
    \end{align*}
\end{proposition}
\begin{proof}
    \begin{align*}
        \dol \theta = \dol(h^{-1} \dol \theta) = - (\dol h^{-1})(\dol h) = - h^{-1} (\dol h) h^{-1} (\dol h) = - \theta^2
    \end{align*}
    where we used $\dol(h^{-1} h) = (\dol h^{-1})h + h^{-1}(\dol h) = 0$. This directly gives
    \begin{align*}
        \Theta = \d \theta + \theta^2 = \dol \theta + \bar \dol \theta + \theta^2 = \bar \dol \theta
    \end{align*}
    The rest of the assertions follow by $\d \Theta = [\Theta, \theta]$.
\end{proof}

\begin{corollary}
    For a holomorphic line bundle $E \rightarrow X$ with Hermitian metric and its canonical connection, we have
    \begin{align*}
        \Theta = \bar\dol \dol \log h
    \end{align*}
\end{corollary}
\begin{proof}
    For line bundle, $h^{-1} = \frac1h$ and we may write
    \begin{align*}
        \Theta = \bar\dol (h^{-1} \dol h) = \bar\dol \dol(\log h)
    \end{align*}
\end{proof}

\begin{lemma} \label{linebundle curvature addition}
    For Hermitian line bundles $L_1, L_2$ endowed with canonical connection and curvature,
    $$\Theta_{L_1 \otimes L_2} = \Theta_{L_1} + \Theta_{L_2}$$
\end{lemma}
\begin{proof}
    For frames $f_\alpha, f_\beta$, let metric $h_j$ ($j=1,2$) have local forms $h_{j}(f_\alpha) = |g_{j, \beta \alpha}|^2 h_j (f_\beta)$ where $g_{j, \beta \alpha}$ is transition function. ($f_\alpha = g_{j, \alpha \beta}f_\beta$ and $h_{j}(f_\alpha) = \langle f_\alpha, f_\alpha \rangle^2$). As transition function for $L_1 \otimes L_2$ is given by $\{ g_{1,\alpha \beta} g_{2, \alpha \beta} \}$, $h_1 h_2$ is a metric on $L_1 \otimes L_2$ and thus we have
    \begin{align*}
        \Theta_{L_1} &= \bar\dol\dol \log h_1 \\
        \Theta_{L_2} &= \bar\dol\dol \log h_2 \\
        \implies \Theta_{L_1 \otimes L_2} &= \bar\dol\dol \log h_1 h_2 = \Theta_{L_1} + \Theta_{L_2}
    \end{align*}
\end{proof}

\textbf{Example. Fubini-Study Metric}

Fubini-Study metric on projective space is used, for example, later to prove positivity property of blow-ups, which is in turn used to prove Kodaira embedding theorem.

For open $U\subset \PP^n$, consider a lift $Z:U \rightarrow \CC^{n+1}-\{0\}$ at each fiber ($\pi \circ Z=\text{id}_{\PP^n}$ where $\pi:\CC^{n+1}-\{0\}\rightarrow \PP^n$). Since $Z$ is a holomorphic function, define
\begin{align*}
\omega = \frac{i}{2\pi} \dol \bar\dol \log ||Z||^2
\end{align*}
Alternative choice of $Z$ doesn't change this form, since for alternative lift $Z'= Z\cdot f$, 
\begin{align*}
	&\dol\bar\dol \log ||Z f||^2 \\
=	&\dol\bar\dol \log||Z||^2 +\dol\bar\dol\log f + \dol\bar\dol\log\bar f\\
=	&\dol\bar\dol \log||Z||^2 -\bar\dol\dol\log f + \dol\bar\dol\log\bar f\\
=	&\dol\bar\dol \log||Z||^2
\end{align*}
Thus we get a well-defined $(1,1)$-form, and thus from this we get a metric:
\begin{align*}
\sum h_{jk} \d z_j \wedge \d z_k \mapsto \sum h_{jk} \d z_j \otimes \d z_k
\end{align*}

Now we define Chern class using curvature. It is initially locally defined using the following formula:

\begin{definition}
Given a vector bundle $E$ with connection $D$, consider the following expression:
\begin{align*}
\left|I+\frac{\sqrt{-1}}{2\pi}\Theta(f)\right|
\end{align*}
which is defined per frame $f$ at local trivial set $U$ via curvature matrix $\Theta(f)$, and is computed by considering wedge $\wedge$ as product. Then the resulting $2k$-forms are called $k$-th \emph{Chern form}, and their de Rham cohomology classes are each called $k$-th \emph{Chern class} $c_{k}$. The formal sum of Chern classes (as element of de Rham cohomology ring) is called the \emph{total Chern class} $c$.
\end{definition}

We claim that this expression gives rise to well-defined (global) differential forms. We will also show that Chern forms, as de Rham cohomology classes, are independent of the connection used. This allows one to speak of the Chern class of a vector bundle.

\begin{proposition}
    Chern class is well-defnied, closed, and invariant of connection used.
\end{proposition}
\begin{proof}
    Using change of frame $A$,
    \begin{align*}
    \left|I+\frac{\sqrt{-1}}{2\pi}A^{-1}\Theta(f)A\right| &= \left|A^{-1}\left(I+\frac{\sqrt{-1}}{2\pi}\Theta(f)\right)A\right|
    \\&= |A|^{-1}\left|I+\frac{\sqrt{-1}}{2\pi}\Theta(f)\right||A|
    \\&= \left|I+\frac{\sqrt{-1}}{2\pi}\Theta(f)\right|
    \end{align*}
    and thus Chern class is independent of frame. Thus we also see that over two different overlapping frames, Chern class is well-defined on the intersection, implying that Chern forms are well-defined global differential form.
    
    To prove that each $2k$-form is closed, we will write each $2k$-form as a polynomial of \emph{traces of $\Theta^k$} and prove that the traces are closed instead. Such identity is proven first for diagonal matrices and is extended to general matrices by denseness of diagonalizable matrices in $\mathbb C^{n^2}$. Closedness of traces is an easy exercise of identity manipulation.
    
    We first prove the following lemma:
    \begin{lemma}
    Suppose characteristic polynomial of a matrix $A$ is given by $c_0 \lambda^n + c_1 \lambda^{n-1} + \cdots + c_n$. Then, defining $b_k=tr(A^k)$, we have
    \begin{align*}
    1b_0c_1 + b_1 c_0 &= 0
    \\2b_0c_2 + b_1 c_1 + b_2 c_0 &= 0
    \\3b_0c_3 + b_1 c_2 + b_2 c_1 + b_3c_0 &= 0
    \\ &\vdots
    \\ nb_0 c_n + b_1 c_{n-1} + \cdots + b_n c_0 &= 0
    \end{align*}
    Thus we can recursively write $c_k$ in terms of polynomial of $b_l$ and vice versa.
    \end{lemma}
    \begin{proof}
    
    This identity is firstly proven relatively easily for diagonal $A$. Then, we can use the fact that diagonalizable matrices form a dense open subset of the set of all matrices to show that it holds for any $A$. 
    
    To show that the identity is true for diagonal matrices, let $A$ have diagonal entries $a_1, \cdots a_n$. Then $b_k = \sum a_i^k$ and $c_k = (-1)^k \sum a_{i_1} \cdots a_{i_k}$, and the identity
    \begin{align*}
    kc_k + b_1 c_{k-1} + \cdots + b_{k-1} c_1 + b_k = 0
    \end{align*}
    translates to:
    \begin{align*}
    &(-1)^k k \left[ \sum a_{i_1} \cdots a_{i_k} \right] 
    \\ +& (-1)^{k-1} \left[ (\sum a_i^1) (\sum a_{i_1} \cdots a_{i_{k-1}})\right]
    \\ +& (-1)^{k-2} \left[ (\sum a_i^2) (\sum a_{i_1} \cdots a_{i_{k-2}})\right]
    \\ +& \cdots
    \\ +& (-1)^{1} \left[ (\sum a_i^{k-1}) (\sum a_{i_1})\right]
    \\ +& (-1)^{0} \left[ (\sum a_i^{k})\right] = 0
    \end{align*}
    This identity looks somewhat daunting at first, but one can consider its summands individually: once the expression is fully expanded, it can only have monomials of the form $a_{i_1}^m a_{i_2} \cdots a_{i_l}$, which can be grouped as either $a_{i_1}^m (a_{i_2} \cdots a_{i_l})$ or $a_{i_1}^{m-1} (a_{i_1} a_{i_2} \cdots a_{i_l})$. Two of them each come from exactly one square bracket of the expression, and due to the alternating signs, their single appearances cancel out exactly to produce 0 in the end. The exceptional $k$ tagged to $c_k$ results from different symmetry for that expression alone.
    
    If the identity is true for matrix $A$, then it is also true for $M^{-1}AM$ since 
    \begin{align*}
    |\lambda I - M^{-1}AM| &= |M^{-1}||\lambda I - A | |M| = |\lambda I - A|
    \\tr((M^{-1}A M)^k) &= tr(M^{-1}A^k M) = tr(A^k)
    \end{align*}
    imply that $b_k, c_k$ are identical for similar matrices. Therefore, the identity is proven for all diagonalizable matrices.
    
    Now we prove that diagonalizable matrices form a dense subset of $\mathcal M_{n,n}$, seen as $\mathbb C^{n^2}$. Observe that a polynomial $(x-a_1) \cdots (x-a_n) = x^n + c_{n-1}x^{n-1} + \cdots + c_0$ has multiple roots iff $\Delta = \prod_{i<j} (a_i - a_j)^2=0$. Since $\Delta$ is a symmetric polynomial in $\{a_i\}$, it is a polynomial in elementary symmetric polynomials, i.e. $\{c_k\}= \{(-1)^k \sum a_{i_1} \cdots a_{i_k}\}$. Since $\{c_i\}$ are polynomials in matrix entries themselves, we see that $\Delta$ is a polynomial in matrix entries. In other words, if we regard $A$ as an element of $\mathbb C^{n^2}$, the locus of $A$ such that $|\lambda I - A|$ has multiple roots is a zero set of a polynomial $\Delta$. Now it's well known that complement of a zero set of a polynomial is dense; if not, there is a ball contained in $V(\Delta)$, for which we can fix all but one variable of $A$ to move a $\epsilon$-interval and still have $A$ lie on the ball. But that's impossible since a monovariate polynomial can't be constantly zero along an interval! (unless the polynomial doesn't contain that variable, but we can a priori choose an entry of $A$ that appears at least once in $\Delta$) Thus by contradiction we see that $V(\Delta)$'s complement is indeed dense in $\mathbb C^{n^2}$ as desired.
    
    Let's show that denseness of diagonalizable matrices imply that the identities hold for all matrices. Suppose $A$ is any matrix and $(A_1, A_2, \cdots )$ is a sequence of matrices converging to $A$ with respect to some norm (say, $L^\infty$ norm of maximum of difference of entries converging to zero). If $(b_k(M),c_k(M))$ are $(b_k,c_k)$ calculated for $M$ (considered as functions of entries of $M$), then we can consider also
    \begin{align*}
    \xi_k(M) = k c_k(M) b_0(M) + \cdots + c_0(M) b_k(M)
    \end{align*}
    Since $(b_k(M),c_k(M))$ are degree-$k$ homogeneous polynomials in the entries of $M$, $\xi_k(M)$ is a degree-$k$ polynomial in entries of $M$. Since polynomials are continuous functions, the convergence of $\{A_i\}$ to $A$ implies the convergence of $\{\xi_k(A_i)\}=\{0\}$ to $\xi_k(A)$, and thus $\xi_k(A)=0$ as desired.
    \end{proof}
    
    $b_k:= tr(\Theta^k)$ are closed because 
    \begin{align*}
    &d\Theta = d(d\theta + \theta\wedge\theta) = d^2\theta + d\theta\wedge \theta - \theta \wedge d\theta = 0+\Theta\wedge \theta - \theta\wedge\Theta
    \\\implies & d(\Theta^k) = [d\Theta \wedge \Theta \wedge \cdots \wedge \Theta] + [\Theta \wedge d\Theta \wedge \cdots \wedge \Theta] + \cdots + [\Theta \wedge \cdots \wedge \Theta \wedge d\Theta]
    \\&= k(d\Theta \wedge \Theta \wedge \cdots \wedge \Theta) = k(\Theta \wedge \theta \wedge \Theta^{k-1} - \theta \wedge \Theta \wedge \Theta^{k-1})
    \\\implies & d b_k = d(tr(\Theta^k)) =  tr(d(\Theta^k)) = k [tr(\Theta \wedge \theta \wedge \Theta^{k-1}) - tr(\theta \wedge \Theta \wedge \Theta^{k-1})] = 0
    \end{align*}
    (where we used $tr(AB)=tr(BA)$ at the end.) Thus by the lemma,
    \begin{align*}
    d(c_k) &= d(\text{polynomial in }b_l)
    \\&=\text{sum of monomials that each contain at least one }d(b_l)
    \\&= 0
    \end{align*}
    as desired.
    
    Lastly we prove that the cohomology classes represented by Chern classes are independent of the connection used. With change of connection, the Chern forms turn out to change precisely by a coboundary, thus preserving their de Rham cohomology classes. This will be proven via interpolating two given connections linearly and expressing the difference in Chern forms via a integral.
\end{proof}

\subsection{Line Bundles}

\subsubsection{Line Bundles and Chern Class}

Chern class is more than a litmus test for non-isomorphic vector bundles. Sometimes, it classifies the bundle completely. Consider following results:

\begin{theorem}
Smooth line bundles on a smooth manifold are uniquely and always determined by their Chern classes in $H^2(X,\mathbb Z)$.
\end{theorem}

\begin{theorem}
Holomorphic line bundles on a complex manifold are not uniquely determined by Chern class, but for each cohomology class reprsented by a $(1,1)$-form, there is at least one line bundle whose Chern class is same as the given cohomology class.
\end{theorem}

\begin{theorem}
Holomorphic line bundles on $\mathbb{CP}^n$ are precisely (isomorphism classes of) $\{H^n\}_{n\in \mathbb Z}$, where $H$ is the hyperplane section bundle and $H^{-n}=(H^*)^n$.
\end{theorem}

These theorems are proven using Chern class of line bundles (which is actually just curvature form, up to constant). The classification of holomorphic line bundles on projective space is powered by the fact that  of holomorphic function sheaf must be trivial for $\mathbb{CP}^n$. This fact can be proven with help of Hodge Decomposition and Serre Duality, which will be proven much later. We'll prove the result here too, assuming the vanishing of cohomology.

Before we go any further, let's discuss a neat computational tool to deal with line bundles:

\begin{proposition}
Holomorphic line bundles are completely classified by the set
\begin{align*}
H^1(X,\mathcal O^\times)
\end{align*}
\end{proposition}
\begin{proposition}
Smooth line bundles are completely classified by the set
\begin{align*}
H^1(X,\mathcal E^\times)
\end{align*}
\end{proposition}
\begin{proof}
Consider that each bundle is completely determined by an open cover and transition functions on intersections, and that transition functions are required to be nonsingular and thus nonvanishing as a scalar. Therefore, each bundle's transition functions determine a collection of nonvanishing holomorphic functions at intersections of an open cover, determining a cocycle of Cech cohomology group. For smooth line bundles, simply replace ``holomorphic'' with ``smooth'' in the proof.
\end{proof}

With this result in mind, we just need to work with $H^1(X,\mathcal O^\times)$ and $H^1(X,\mathcal E^\times)$ to classify holomorphic line bundles. Given that we are to care about sheaf cohomology of $\mathcal O^\times$, the ``exponential sheaf sequence" becomes extremely useful. We are considering the exact sequence of sheaves
\begin{align*}
&0\rightarrow \underline{\mathbb Z} \xrightarrow{\subset} \mathcal O \xrightarrow{e^{2\pi i \cdot}} \mathcal O^\times \rightarrow 0\\
&0\rightarrow \underline{\mathbb Z} \xrightarrow{\subset} \mathcal E \xrightarrow{e^{2\pi i \cdot}} \mathcal E^\times \rightarrow 0
\end{align*}
and the homology long exact sequence produced by it:
\begin{align*}
& 0\rightarrow \underline{\mathbb Z}(X) \rightarrow \mathcal O(X) \rightarrow \mathcal O^\times(X) \\
\rightarrow & H^1(X,\underline{\mathbb Z}) \rightarrow H^1(X,\mathcal O) \rightarrow H^1(X,\mathcal O^\times) \\
\rightarrow &H^2(X,\underline{\mathbb Z}) \rightarrow H^2(X,\mathcal O) \rightarrow H^2(X,\mathcal O^\times) \\ \rightarrow & \cdots
\end{align*}
and
\begin{align*}
& 0\rightarrow \underline{\mathbb Z}(X) \rightarrow \mathcal E(X) \rightarrow \mathcal E^\times(X) \\
\rightarrow & H^1(X,\underline{\mathbb Z}) \rightarrow H^1(X,\mathcal E) \rightarrow H^1(X,\mathcal E^\times) \\
\rightarrow &H^2(X,\underline{\mathbb Z}) \rightarrow H^2(X,\mathcal E) \rightarrow H^2(X,\mathcal E^\times) \\ \rightarrow & \cdots
\end{align*}
here, the latter is easier to deal with, and directly yields the classification of smooth line bundles, since $\mathcal E$ is a fine sheaf, and thus an acyclic sheaf, meaning that $H^r(X,\mathcal E)=0$ identically, making the long exact sequence simply
\begin{align*}
& 0\rightarrow \underline{\mathbb Z}(X) \rightarrow \mathcal E(X) \rightarrow \mathcal E^\times(X) \rightarrow H^1(X,\underline{\mathbb Z}) \rightarrow 0 \\
&0 \rightarrow H^1(X,\mathcal E^\times)
\rightarrow H^2(X,\underline{\mathbb Z}) \rightarrow 0 \\
&0 \rightarrow H^2(X,\mathcal E^\times) \rightarrow H^3(X,\underline{\mathbb Z}) \rightarrow 0 \\
& \cdots
\end{align*}
therefore having equality:
\begin{align*}
H^r(X,\mathcal E^\times) \cong H^{r+1}(X,\underline{\mathbb Z})
\end{align*}
and especially that the Bockstein homomorphism of the long exact sequence completely classifies line bundle through $H^2(X,\underline{\mathbb Z})\cong H^2(X,\mathbb Z)$.

Meanwhile, we can't do the same thing with holomorphic line bundle since $\mathcal O$ is not in general an acyclic sheaf.

\subsubsection{Line Bundles and Divisors}

Given a complex manifold $X$, an \emph{analytic hypersurface} on it is a subset that is locally given as a zero locus of a complex analytic function. A \emph{divisor} is a formal finite linear combination of analytic hypersurfaces. This generalizes the notion of divisor on a Riemann surface, a finite linear combination of points.

Divisors turn out to be closely related to \emph{line bundles}, which are defined as vector bundles of rank 1. To give the relationships concisely, we first describe divisors and line bundles as global sections of some sheaves.

\begin{proposition}
    Divisors are characterized by $H^0(X,\mathcal M^\times / \mathcal O^\times)$, line bundles are characterized by $H^1(X,\mathcal O^\times)$, and they are connected by Bockstein homomorphism resulting from cohomology long exact sequence 
    \begin{align*}
    & 0 \rightarrow \mathcal O^\times \rightarrow \mathcal M^\times \rightarrow \mathcal M^\times / \mathcal O^\times \rightarrow 0 \\
    \implies & 0 \rightarrow H^0(X,\mathcal O^\times) \rightarrow H^0(X, \mathcal M^\times) \rightarrow H^0(X, \mathcal M^\times / \mathcal O^\times) \rightarrow H^1(X,\mathcal O^\times) \rightarrow \cdots
    \end{align*}
\end{proposition}
\begin{proof}
    \begin{itemize}
    \item \textbf{Divisor = Element of $H^0(X, \mathcal M^\times / \mathcal O^\times)$}: A divisor can be thought as a zero locus of locally defined locally defined functions with appropriate multiplicity. Conversely, when given locally defined meromorphic functions with quotients in $\mathcal O^\times$ (the patchwork definition arises from sheafification of quotient sheaf $\mathcal M^\times / \mathcal O^\times$) we can consider their zero loci and define that to be the associated divisor. The ``zero locus'' notion is made precise through considering vanishing order $\text{ord}_V(f)$.
    \item \textbf{Line Bundle = Element of $H^1(X, \mathcal \mathcal O^\times)$}: A line bundle gives rise to transition functions at overlaps of local trivializations, which are nonzero scalar ($\text{GL}_1(\mathbb{C}) = \mathbb {C}^\times$). This gives an element of (\v Cech) cohomology group $H^1(X,\mathcal O^\times)$. Conversely, such transition functions give rise to a line bundle by gluing $U_\alpha \times \mathbb {C}$ using transition functions.
    \item \textbf{Divisor induces a Line Bundle}: For local defining functions of a divisor, multiply them out with appropriate multiplicities in the divisor. This gives a scalar (meromorphic) function per a patch. Now take quotient of the scalar functions over patch overlaps. The ``poles'' of meromorphic function will cancel out since the functions are defined from same hypersurfaces.
    \end{itemize}
\end{proof}

\textbf{Example: Hyperplane Bundle.} A \emph{hyperplane bundle} of projective space $\mathbb{P}^n$ is $[H]$ where $H$ is any hyperplane. This turns out to be the same regardless the choice of the hyperplane. For example, let $H$ be given by
\begin{align*}
f = a_0 x_0 + \cdots + a_n x_n = 0
\end{align*}
Let $U_j=\{[x_0 : \cdots x_n] | x_j \neq 0\}\subset\mathbb{P}^n$. In $U_j$, the function $f/x_j$ is well-defined and smooth. Further, $f/x_j = 0$ defines the hyperplane $H$ in $U_j$. Therefore, the transition function associated to line bundle $[H]$ is 
\begin{align*}
g_{\alpha \beta} = \frac{f/x_\alpha}{f/x_\beta} = \frac{x_\beta}{x_\alpha}
\end{align*}
From this transition function data we can construct the bundle $[H]$. Note that $a_j$ are missing from the transition function data, and thus that the choice of $H$ doesn't matter at all.

\textbf{Example: Universal Bundle.} The \emph{universal bundle} or \emph{tautological bundle} is dual to the hyperplane bundle. It is named ``universal'' because any vector bundle on any manifold is pullback of the universal bundle on a Grassmannian, with $\mathbb{P}^n$ being a special case of Grassmannians. It is named ``tautological'' because an explicit construction follows by attaching to each point $\{[\lambda x_0 : \cdots : \lambda x_n]\} \in \mathbb{P}^n$ the line $\{(\lambda x_0, \cdots ,\lambda x_n)\} \subset \mathbb{C}^{n+1}$, which is the point itself written differently. Local trivializations are given over each $U_j$: $(x_0, \cdots x_n) \mapsto x_j$. Then, transition function is given by 
\begin{align*}
g_{\alpha\beta} = \frac{x_\alpha}{x_\beta}
\end{align*}
This is precisely reciprocal of $\frac{x_\beta}{x_\alpha}$, which is the transition function for the hyperplane bundle. Therefore, the universal bundle is also $[H]^*=[-H]$.

Given a line bundle $L$, let's compute transition functions associated to $L^{\otimes k} := L \otimes \cdots \otimes L$ and $L^*$. 

\begin{proposition}
    Given complex line bundles $L_1, L_2$ on $X$ with transition functions $g_{1,\alpha \beta}, g_{2, \alpha \beta} : U_{\alpha} \cap U_\beta \rightarrow \CC^\times$, the transition functions associated to $L_1 \otimes L_2$ and $L_1^*$ are $g_{1,\alpha \beta} \cdot g_{2, \alpha \beta}$ and $g_{1, \alpha \beta}^{-1}$ respectively.
\end{proposition}
\begin{proof}
    We already demonstrated in the computation of transition function for cotangent bundle that transition function for $L_1^*$ is $g_{1, \alpha \beta}^{-1}$. 
    
    Let's prove the tensor product relation. Abbreviate $g_1 = g_{1,\alpha \beta}(x), g_2 = g_{2, \alpha \beta}(x)$ below (where $x \in U_{\alpha} \cap U_\beta$).
    \begin{align*}
        & (g_1 \otimes g_2)(v_1 \otimes v_2) = g_1(v_1) \otimes g_2(v_2) \\
        \implies & (g_1 \otimes g_2)' (v) = (g_1(1) \cdot g_2(1)) \cdot v = (g_1g_2)(v)
    \end{align*}
    where we associate to each $f \in \GL(\CC \otimes_\CC \CC)$ the map $f' \in \GL(\CC)$ given by identification $\CC \otimes_\CC \CC \xrightarrow{\sim} \CC; v_1 \otimes v_2 \mapsto v_1v_2$. This proves the desired relation.
\end{proof}

\textbf{Example: Global Sections of $[mH]$ on $\PP^n$.} Let's calculate global sections of tensor power of hyperplane bundle. Interestingly, it turns out that
\begin{align*}
\mathcal O(\PP^n , [mH]) \cong \{\text{Homogeneous polynomials of degree $m$ in $n$ variables}\}
\end{align*}
Firstly, let's check what functions we need to assign. $\PP^n$ trivializes over $U_0 \cdots U_n$ where $U_j = \{(z_0 , \cdots z_n) | z_j \neq 0\} \subset \PP^n$. Each holomorphic section $s$ yields a holomorphic function $s_j \in \mathcal O(U_j)$ such that $s_k=g_{jk}s_j = (\frac{z_j}{z_k})^m s_j$. 

Now suppose $s_0 \in \mathcal O(U_0)$. Suppose by local coordinates of $U_0$, $s_0$ translates to $f \in \CO (\tilde U_0)$, where $\tilde U_0 \cong \CC^n$ is the local coordinate for $U_0$. Write the local coordinates as $(w_{01}, \cdots w_{0n})$. Now when we translate this to local trivialization over local coordinates of $U_1$, we get 
\begin{align*}
w_{01}^{-m} f(w_{01}, \cdots w_{0n})
\end{align*}
If we write local coordinates of $U_1$ to be $(w_{10}, w_{12}, \cdots w_{1n})$, then since the point $(w_{01}, \cdots w_{0n})$ corresponds to $(w_{10}, \cdots w_{1n}) = (\frac1{w_{01}}, \frac{w_{02}}{w_{01}}, \cdots \frac{w_{0n}}{w_{01}})$, we see that $f$ translates to 
\begin{align*}
w_{10}^m f(\frac1{w_{10}}, \frac{w_{12}}{w_{10}}, \cdots \frac{w_{1n}}{w_{10}})
\end{align*}
Now observe that if power series expansion of $f$ (at some point) contains a monomial of degree greater than $m$, then this translated function will surely have a pole in $w_{10}$ and not be holomorphic. This can't happen, and thus $f$ has to be a polynomial of degree $\leq m$.

Write $w_{0j} = \frac{z_j}{z_0}$. Then our $f$ can be viewed as $z_0^{-m}F$ where $F$ is a homogeneous polynomial of degree $m$ in $z_0, \cdots z_n$. In our above notation, monomial $f = w_{01}^{j_1} \cdots w_{0n}^{j_n}$ translates into $w_{10}^{m-\sum j_l} w_{12}^{j_2} \cdots w_{1n}^{j_n}$. This translation corresponds exactly by also writing $w_{1j} = \frac{z_j}{z_1}$. Therefore, we see that each holomorphic section of $\mathcal O(mH)$ corresponds to a homogeneous polynomial of degree $m$ in $z_0, \cdots z_n$. The inverse correspondence is established by associating $z_j^{-m}F(z_0, \cdots z_n)$ to each $U_j$.

\textbf{Example: Canonical Bundle.}

Canonical bundle of a complex manifold $X$ is $K_X := \Omega_X^n$, the bundle of holomorphic $n$-forms. We compute its transition function:
\begin{align*}
    \d x_j &= \sum_{k=1}^n \frac{\partial (y \circ x^{-1})_j}{\partial t_k} \d y_k \\
    \implies \d x_1 \wedge \cdots \wedge \d x_n &= \bigwedge_{j=1}^n \left(\sum_{k=1}^n \frac{\partial (y \circ x^{-1})_j}{\partial t_k} \d y_k \right) \\
    &= \left(\sum_{\sigma} (-1)^{\on{sign}(\sigma)} \prod_{j=1}^n \frac{\partial (y \circ x^{-1})_j}{\partial t_{\sigma(j)}}\right) (\d y_1 \wedge \cdots \wedge \d y_n) \\
    &= |J_{y \circ x^{-1}}| (\d y_1 \wedge \cdots \wedge \d y_n)
\end{align*}
where $J_{y \circ x^{-1}}$ is the Jacobian for $y \circ x^{-1}$.

The following give more concrete information about line bundles associated to divisors.

\begin{proposition}\label{globalmero}
For divisor $D$, there is a global meromorphic section $s_D \in \mathcal M (M,[D])$ such that 
\begin{align*}
(s_D)=D
\end{align*}
\end{proposition}
\begin{proof}
We use local defining functions of $D$ to be local meromorphic functions that define $s_D$. We need that quotients of such functions on overlaps equal quotients of local defining functions of $D$, which is tautology.
\end{proof}

\begin{proposition}\label{bundleglobalmero}
A line bundle is of the form $[D]$ iff there is a global meromorphic section.
\end{proposition}
\begin{proof}
We proved already in Proposition~\ref{globalmero} that $[D]$ has a global meromorphic section. Now we establish the converse. If $L$ has a global meromorphic section $s$, then we claim that $L=[(s)]$. Under scrutiny, however, this is tautology since 
\begin{align*}
	&\text{ Transition function of $[(s)]$} \\
=	&\text{ Quotient of local defining function of $(s)$} \\
=	&\text{ Transition function of $L$}
\end{align*}
In the last step, we used the description of a meromorphic section by local meromorphic functions with quotients being transition functions $s_\beta = g_{\alpha \beta} s_\alpha$.
\end{proof}

\begin{proposition}\label{merofromquotient}
Given two meromorphic sections $\xi,\eta \in \mathcal M(M, L)$, $\xi/\eta$ given by quotient of local meromorphic functions defining $\xi,\eta$ is a well-defined meromorphic function in $\mathcal M(M)$.
\end{proposition}

\begin{proposition}~\label{restrictioncohomology}
For compact manifold $M$ and a smooth hypersurface $V\subset M$, the following sheaf sequence is exact:
\begin{align*}
0 \rightarrow \CO (E \otimes [-V]) \rightarrow \CO(E) \rightarrow \CO(E|_V) \rightarrow 0
\end{align*}
In other words, $\CO(E\otimes [-V])$ encodes sections of $E$ that vanish at $V$.
\end{proposition}

\begin{proof}
The idea is that we take $\xi \otimes s \in \CO(M,E) \otimes \CO(M,[-V])$ and get $\xi \cdot \frac{s}{s_{-D}} \in \CO(E)$. This map factors through sheafification, and thus defines a map $\CO(E \otimes [-V]) \rightarrow \CO(E)$.
\end{proof}

\pagebreak
\section{Hodge Theory}

In this section, we prove that each de Rham cohomology class contains a unique harmonic differential form, and that we can replace the whole discussion of de Rham cohomolgy by the discussion of harmonic forms:
\begin{align*}
H^r(M,\mathbb{C}) \cong \mathcal H^r(M)
\end{align*}
As explained in the Introduction, this is the most central theorem in the whole text that allows proof of many other theorems like Hodge decomposition and Kodaira embedding theorem. 

The proof benefits from introducing Sobolev spaces on vector bundles. The primary motivation is that inner product / norm allows us to use Hilbert space / Banach space theory in this context. In the end all the relevant details become hidden, however.

\subsection{Sobolev Space and Differential Operators on Vector Bundles}

First we introduce some preliminary notions.

Roughly, Sobolev norm is a measure of smoothness. The firsthand definition doesn't seem to say much about smoothness, but we will soon prove that
\begin{align*}
||f||_s^2 = \sum_{|\alpha|\le s} ||D^\alpha f||_{L^2}^2
\end{align*}
which tells us that the squared norm is sum of squared $L^2$-norms of derivatives up to a certain order.
\begin{definition}
Sobolev norm of a compactly supported smooth complex function $f:\mathbb R^n \rightarrow \mathbb C$ is defined as
\begin{align*}
||f||_s^2 = \int |\hat f|^2 (1+|y|^2)^s dy
\end{align*}
Sobolev norm is extended to $\mathbb{C}^m$-valued functions by taking scalar Sobolev norm of Euclidean norm of the vectors.

Sobolev norm is also defined for vector bundle $E$ over a compact complex manifold $M$, by adding vector-valued Sobolev norms over trivializing cover of the vector bundle with partition of unity:
\begin{align*}
||\xi||_{s,E}^2 = \sum_\alpha ||\varphi^*_\alpha (\rho_\alpha \cdot \xi)||_{s,\mathbb{R}^n}
\end{align*}
Here, we are assuming that $E$ is trivial and $M$ is Euclidean over finite open cover $\{U_\alpha\}$. $\varphi_\alpha$ is such trivialization of $E$: $\varphi_\alpha : E|_{U_\alpha} \rightarrow \tilde{U}_\alpha \times \mathbb{C}^m$ and $\varphi_\alpha^*$ is the pullback $\varphi_\alpha^* : \mathcal E(U_\alpha,E) \rightarrow (\mathcal E(\tilde{U}_\alpha))^m$. $\{\rho_\alpha\}$ are partition of unity.
\end{definition}

Note that Sobolev $0$-norm $||\cdot ||_0$ is $L^2$-norm.

For multi-index $\alpha=(\alpha_1, \cdots \alpha_n)$, let $y^\alpha = y_1^{\alpha_1} \cdots y_n^{\alpha_n}$ and $D^\alpha = {i^{-|\alpha|}} (\frac{\partial}{\partial x_1})^{\alpha_1} \cdots (\frac{\partial}{\partial x_n})^{\alpha_n}$. 

\begin{proposition}
For $f:\mathbb{R}^n \rightarrow \mathbb C$ and $s\in\mathbb Z$, 
\begin{align*}
|| f ||_s^2 = \sum_{|\alpha|\le s} ||D^\alpha f||_0^2
\end{align*}
\end{proposition}
\begin{proof}
We use these two simple observations:
\begin{align*}
&\widehat{D^\alpha f}(y) = y^\alpha \hat{f}(y) \\
&(1+|y|^2)^s = (1+y_1^2 + \cdots + y_n^2)^s = \sum_{|\alpha|\le s} |y^\alpha|^2
\end{align*}
Now,
\begin{align*}
||f||_s^2 &= \int |\hat f|^2 (1+|y|^2)^s dy
=\int |\hat f|^2 \sum_{|\alpha|\le s} |y^\alpha|^2 dy
=\sum_{|\alpha|\le s} \int |y^\alpha \hat f|^2 dy
=\sum_{|\alpha|\le s} ||\widehat{D^\alpha f}||_0^2
=\sum_{|\alpha|\le s} ||{D^\alpha f}||_0^2
\end{align*}
\end{proof}

Let $W^s(E)$ be the completion of $\mathcal E(M,E)$ with respect to Sobolev $s$-norm. The following lemma gives basic smoothness relation for Sobolev spaces.

\begin{lemma}[Sobolev]
Elements of $W^s(E)$ are at least $s-\lfloor \frac{n}{2} \rfloor - 2$ times continuously differentiable.
\end{lemma}

We also develop notions of differential operators and pseudodifferential operators on vector bundles. A differential operator is a map of global sections of vector bundles which restricts to differential operator in the ordinary sense. The chief examples of differential operators are $\d,\partial, \Delta$. Meanwhile, construction of pseudodifferential operators is more involved. Ultimately this is necessary since we will be finding pseudo-inverses to $\Delta$, and this cannot be done if we only have differential operators.

More precisely, we will be assigning ``symbol'' to each pseudodifferential operator, and we will be looking for operators that yield a given symbol. But this doesn't work if we only have differential operators. In other words,
\begin{align*}
\text{PDiff}_m \rightarrow \text{Symb}_m
\end{align*}
is surjective, while
\begin{align*}
\text{Diff}_m \rightarrow \text{Symb}_m
\end{align*}
is not.

Let's first briefly discuss what we mean by differential operator in the ordinary sense. The most reasonable model to refer to as ``differential operator'' is certainly expressions like the following:
\begin{align*}
a_1 \frac{\partial}{\partial x}+a_2 \frac{\partial^2}{\partial y\partial z}
\end{align*}
For more general considerations, let's write $D^\alpha = (-i)^{|\alpha|}D_1^{\alpha_1} \cdots D_n^{\alpha_n}$ where $\alpha=(\alpha_1, \cdots \alpha_n)$ is a multi-index, $|\alpha|=\sum \alpha_j$ is its modulus and $D_n = \frac{\partial}{\partial x_n}$. (Here $(-i)^{|\alpha|}$ is normalization to make sure $\widehat{D^\alpha f}(y) = y^\alpha \hat{f}(y)$) Now we can define a linear differential operator (here onwards abbreviated as LDO) as an operator of the following form:
\begin{align*}
\sum_\alpha a_\alpha D^\alpha
\end{align*}
If a LDO only sums over $|\alpha|\le r$, then we say that it has rank $r$.

So far, we only discussed map from $\mathcal E(\mathbb{R}^n)$ to $\mathcal E(\mathbb R^n)$. We can actually go a little farther and define differential operator as a map from $\mathcal E(\mathbb{R}^n)^p$ to $\mathcal E(\mathbb{R}^n)^q$ as
\begin{align*}
& L : \begin{bmatrix} f_1 \\ \vdots \\ f_m \end{bmatrix} \mapsto \begin{bmatrix} L_{11} & \cdots & L_{1m} \\ \vdots & \ddots & \vdots \\ L_{n1} & \cdots & L_{nm} \end{bmatrix} \begin{bmatrix} f_1 \\ \vdots \\ f_m \end{bmatrix} \\
& L_{jk} = \sum_\alpha a_\alpha^{jk} D^\alpha
\end{align*}
In other words, we take linear combinations of LDO and assign them to each coordinate.

Differential operator on vector bundles is simply defined as a linear map on global sections that restricts to LDO on each trivialization.

\begin{definition}
Given vector bundles $E$ and $F$ of rank $m,n$ on smooth $N$-manifold $M$, a differential operator from $E$ to $F$ is a map
\begin{align*}
L: \CE (E,M) \rightarrow \CE (F,M)
\end{align*}
such that for each $U\subset M$ over which $E$,$F$ trivializes and $M$ is Euclidean, $\widetilde L_U$ is a linear differential operator. Here, $\widetilde L_U$ is the effect of $L|_U$ on trivializations: the identification $\CE (E, U) \cong \CE (\RR ^N)^m, \CE (F, U) \cong \CE (\RR ^N)^n$ gives
\[
\begin{array}{ccc}
\CE (E,U) & \xrightarrow{L|_U} & \CE (F,U) \\
\downarrow &  & \downarrow \\
\CE (\RR ^N)^m & \xrightarrow{\widetilde{L}_U} & \CE (\RR ^N)^n
\end{array}
\]
A differential operator is said to have order $r$ if each local LDO has order at most $r$.
\end{definition}

\medskip

\textbf{Example: Exterior Derivative.} Chief examples of differential operator between vector bundles are exterior derivatives $\d,\partial,\bar\partial,$ and Laplacian $\Delta$. Let's look at $\d$ for example. Over local coordinate of a manifold, $\d$ acts as:
\begin{align*}
\d (f \d x \wedge \d y + g \d x \wedge \d z + h \d y \wedge \d z) = (\frac{\dol f}{\dol z} - \frac{\dol g}{\dol y} + \frac{\dol h}{\dol x}) \d x \wedge \d y \wedge \d z
\end{align*}
Thus we see that in this case $\d$ acted as a matrix
\begin{align*}
\d : \begin{bmatrix} f \\ g \\ h \end{bmatrix} \mapsto \begin{bmatrix} \frac{\dol}{\dol z} & - \frac{\dol}{\dol y} &  \frac{\dol}{\dol x} \end{bmatrix} \begin{bmatrix} f \\ g \\ h \end{bmatrix}
\end{align*}
More generally, matrix of $\d$ is $\binom nk$-by-$\binom n{k+1}$.

Local description of $\dol,\bar \dol$ are almost the same. For Laplacian $\Delta$, it locally acts as second-order LDO.

Now we proceed to define symbol of differential operator. A symbol is simply substitution of partial derivatives into formal variables:
\begin{align*}
a_1 \frac{\partial}{\partial x}+a_2 \frac{\partial^2}{\partial y\partial z} \xrightarrow{\text{Symbol}} a_1 X + a_2 YZ
\end{align*}
For formal discussion on vector bundles, however, we will go on a little detour to define it globally.

This effort is necessary since later, we will define symbol for pseudodifferential operators. This discussion is necessary since ultimately the we will prove exactness of the sequence
\begin{align*}
\text{PDiff}_m \rightarrow \text{Symb}_m \rightarrow 0
\end{align*}
and that the kernel of the symbol map is an operator of order $m-1$. This will be the pivotal construction for everything else to follow in the proof of Hodge theorem. 

Action of a linear differential operator can be written as a Fourier transform in the following sense:
\begin{align*}
\widehat{D^\alpha f}(y) = y^\alpha \hat{f}(x) & \implies D^\alpha f(x) = \int y^\alpha \hat f(y) e^{i\langle x,y\rangle} dy \\ 
& \implies (\sum a_\alpha D^\alpha)f (x) = \int (\sum a_\alpha y^\alpha) \hat f(y) e^{i\langle x,y\rangle} dy \\
& \implies p(x,D) f (x) = \int p(x,y) \hat f(y) e^{i\langle x,y\rangle} dy
\end{align*}
Here, $p(x,y)$ is a smooth function in $x$ and polynomial in $y$.
Pseudodifferential operator attempts to generalize this situation by replacing $p$ with a smooth function in $x,y$, but with appropriate growth conditions.

Symbol of a pseudodifferential operator can also be defined. It turns out that every symbol has a PDO giving that symbol. Giving a well-defined symbol map turns out to be a messy task, the proof of which we will omit here.

The following is the most important theorem for pseudodifferential operators in this text:

\begin{proposition}\label{pdiffsymb}
The following map is surjective:
\begin{align*}
\text{PDiff}_m(E,F) \rightarrow \text{Symb}_m(E,F)
\end{align*}
and the kernel of this map are operators of order $m-1$.
\end{proposition}

\subsection{Elliptic Differential Operator and Hodge Theorem}

Here we prove Hodge theorem for elliptic differential operators. 
\begin{definition}
    A pseudodifferential operator $\Delta \in \on{PDiff}_k(E,F)$ is said to be \emph{elliptic} if the assocaiated symbol map $\sigma_k(\Delta)$ is such that $\forall (x, \xi) \in T'(X), \sigma_k(\Delta)(x,\xi): E_x \rightarrow F_x$ is an isomorphism.
\end{definition}

In order to prove Hodge theorem, we establish three key facts about an elliptic differential operator $\Delta$:

\medskip

\textbf{Finiteness.} $\dim \CH^r(X) < \infty$.

\medskip

\textbf{Regularity.} For extension $\Delta_s : W^s \rightarrow W^{s-2}$, if $\Delta_s \xi$ is smooth, then $\xi$ is smooth.

\medskip

\textbf{Existence.} For each differential form $\varphi$ orthogonal to $\CH^r$, there is a unique $\psi$ orthogonal to $\CH^r$ such that $\Delta \psi = \varphi$.

\medskip

Finiteness is necessary to define a projection in Hilbert space $W^0(E) = L^2(E)$. Regularity is used in various situations to show that an element of Sobolev space obtained as a ``weak solution" is actually smooth. (Thanks to this, we don't speak of Sobolev spaces in the statement of Hodge theorem) Existence is used to find Green's operator $G$; $G$ is obtained by applying Banach open mapping theorem, and for conditions of this theorem to be met, we need bijection and existence theorem.

Proofs of all three of these theorems rely on construction of parametrix of $\Delta$, which is a sort of a pseudoinverse. This is precisely where all the previous discussions of pseudodifferential operators become useful. We will first discuss parametrix, then the proof of three theorems, and finally the proof of Hodge theorem.

\begin{definition}
    A \emph{parametrix} of $L\in \text{PDiff}_m(E,F)$ is $P \in \text{PDiff}_m$ such that $L \circ P = P \circ L = I - S$ where $S\in \text{OP}_{-1}(E,F)$. In other words, it is inverse of $L$ up to an operator of degree $-1$.
\end{definition}

The key theorem is that elliptic operators have a parametrix, and this is a direct consequence of Proposition~\ref{pdiffsymb}.

\begin{proposition}
    Elliptic pseudodifferential operator $\Delta$ has a parametrix.
\end{proposition}
\begin{proof}
    Symbol of $\Delta$ is invertible, and to this inverted symbol corresponds a pseudodifferential operator $P$, by Proposition~\ref{pdiffsymb}. Thus
    \begin{align*}
    			& \sigma(P \circ \Delta) = \sigma(P) \circ \sigma(\Delta) = \text{id} \\
    \implies 	& \sigma(P \circ \Delta - I) = 0 \\
    \implies	& P \circ \Delta - I \in \text{OP}_{-1}(E,F)
    \end{align*}
\end{proof}

With this we can easily prove regularity theorem: $\Delta \xi$ smooth $\implies$ $\xi$ smooth. 
\begin{theorem}[Regularity for Elliptic Operator]
For elliptic differential operator $\Delta \in \text{Diff}_m(E,F)$, if for $\xi \in W^s(E)$, $\Delta_s\xi$ is smooth, then $\xi$ is also smooth.
\end{theorem}
\begin{proof}
Essentially, $P\circ \Delta - I$ being operator of degree $-1$ allows one to show that $\xi$ is in $W^s$, then $W^{s+1}$, then $W^{s+2}$, and so on, ultimately showing that by Sobolev lemma, $\xi$ is (infinitely) smooth.

Since $P\circ \Delta = I-S$
\begin{align*}
\xi = (P \circ \Delta - S)\xi = \text{(Smooth Element) } + \text{ (Element of $W^{s+1}$)} \in W^{s+1}(E)
\end{align*}
where $P\Delta \xi$ is smooth since $\Delta \xi$ is smooth. Repeating this, we see that $\xi \in W^k(E)$ for arbitrarily large $k$, and by Sobolev lemma, $\xi$ is smooth.
\end{proof}

Now we proceed to prove finiteness theorem. It directly benefits from existence of parametrix; parametrix allows us to translate the result into ``Hilbert space version'', and the Hilbert space version is actually very simple. It uses the following lemmas from functional analysis:

\begin{lemma}[Rellich]
The inclusion $W^{n+k}(E) \subset W^{n}(E)$ is a compact operator.
\end{lemma}

\begin{lemma}\label{hilbertkernelfinite}
Suppose $S:H\rightarrow H$ is a compact operator of Hilbert spaces. Then $\dim \ker (I-S) < \infty$.
\end{lemma}
\begin{proof}
We use definition of compact operator to show that unit ball is compact, and from this we prove by contradiction that our Hilbert space is finite-dimensional (infinite-dimensional space gives an open cover with no finite subcover).

Closed unit ball $B_1$ is bounded, so for $B_1' = \ker(I-S)\cap B_1 = \{x \in B_1 | S(x) = x\}$, $S(B_1') = B_1'$ is closed and thus compact. This is closed unit ball of $\ker(I-S)$.

Now assume $H$ has an infinite orthonormal basis $\{e_\mu\}$. Cover the unit ball with balls of radius $\frac1{\sqrt 2}$, centered at each point of the sphere. Then there is no finite subcover of this, since each ball can't cover two $e_\mu$ at the same time ($||e_\mu - e_\kappa||^2 = ||e_\mu||^2 + ||e_\kappa||^2 - 2\langle e_\mu,e_\kappa \rangle = 1+1-0=2$). Thus infinite-dimensionality of Hilbert space implies non-compactness of the unit ball, and contrapositively we see that our unit ball is compact.
\end{proof}
\textbf{Remark.} One also sees that when $S$ is compact, $I-S$ is compact and thus $\ker S = \ker I-(I-S)$ is finite-dimensional. 

\textbf{Remark.} This lemma holds if we replace $H$ by Banach space. Also, the cokernel is finite-dimensional.

\begin{theorem}[Finiteness for Elliptic Operator]
Given elliptic differential operator $\Delta \in \text{Diff}_k(E,F)$, its kernel is finite-dimensional.
\end{theorem}
\begin{proof}
Find parametrix of $\Delta$, say $P\in \text{PDiff}_k(F,E)$. Then $\Delta \circ P=P \circ \Delta$, extended to Sobolev space, is a compact operator, since it's operator of degree $-1$, and by Rellich lemma inclusion is a compact operator. By Lemma~\ref{hilbertkernelfinite}, $\ker (P \circ \Delta)_s$ is finite-dimensional, and so is $\ker \Delta_s$. Finally, by Regularity, $\ker \Delta_s = \ker \Delta$.
\end{proof}

We finally establish existence theorem; it provides a bijection where the inverse function will be the Green's operator $G$ in the Hodge theorem. In its proof, we will establish surjectivity via ``solving'' equation in weak sense (in Sobolev space) and then show that it's actually smooth with Regularity. The proof also depends on several results in functional analysis.

\begin{theorem}[Existence for Elliptic Operator]
For elliptic differential operator $\Delta \in \text{Diff}_m(E,F)$, there is bijection
\begin{align*}
L : \CE(X,E) \cap \CH_\Delta^\perp \longleftrightarrow \CE(X,F) \cap \CH_{\Delta^*}^\perp
\end{align*}
where $\Delta^*$ is adjoint to $\Delta$.
\end{theorem}
\begin{proof}
For each $\eta \in \CE(X,F)$ orthogonal to $\CH_{\Delta^*}$, we will find $\xi \in W^s(E)$ such that $L_s(\xi) = \tau$. Regularity then immediately shows that $\xi$ is smooth.

In fact, we will solve the above relation more generally for $\eta \in W^0(F)$. Now, generally closure of range is perpendicular to kernel of transpose, and thus closure of $L_m(W^m(E))$ is $W^0(F) \cap \CH_{\Delta^*}^\perp$. Furthermore, we don't even need to take closure since finite-dimensionality of $\text{coker} L_m$ implies that $L_m(W^m(E))$ is closed. Thus, $L_m(W^m(E)) = W^0(F) \cap \CH_{\Delta^*}^\perp$.

Now we establish surjectivity and injectivity by finding unique solution in $W^m(E) \cap \CH_{\Delta}^\perp$ by projecting from $W^m$ to $\ker L_m$, which is finite-dimensional. 
\end{proof}

We now restate and prove the Hodge theorem presented in the introduction.

\begin{theorem}[Hodge]
There is an isomorphism
\begin{align*}
H^r(M,\mathbb C) \cong \mathcal H^r(M)
\end{align*}
induced from orthogonal projection
\begin{align*}
\mathcal H: \mc E^r(X) \rightarrow \mathcal H^r (X)
\end{align*}
satisfying $I = \mathcal H + \Delta G = \mathcal H + G\Delta$. 

Also, $G(\mathcal H^r)=0$, $\{G,\d,\Delta\}$ commute, and any of these operators composed with $\CH$ gives $0$.
\end{theorem}

\begin{proof}
The isomorphism follows from $I = \CH + \Delta G$ and commutation relations. Thus we will show those first.

Projection $\CH$ is well-defined by finite-dimensionality of the kernel. Now Existence theorem gives a bijection:
\begin{align*}
\Delta_m : W^m(E) \cap \CH_\Delta^\perp \rightarrow W^0(E) \cap \CH_\Delta^\perp
\end{align*}
Banach open mapping theorem tells us that the inverse is continuous too. We simply call that inverse $G$, the Green's operator. 

For now, composing $\Delta$ and $G$ gives identity. By extending $G$ to the rest of $W^0(E)$ by giving value $0$ to elements of $\CH_\Delta$, we get
\begin{align*}
G \circ \Delta = \Delta \circ G = I_{\CH^\perp} = I-\CH
\end{align*}

Now let's prove the identities.

$[\d,\Delta] = [\d^*, \Delta]=0$: for example, $\d \Delta = \d \d^* \d = \Delta \d$ by $\d^2 = 0$. 

$G \CH = \CH G = 0$: this is because $G$ is $0$ on $\CH_\Delta$, and $G$ maps to space orthogonal to $\CH_\Delta$.

$\d \CH = \d^* \CH = 0$: this follows from $\ker \Delta = (\ker \d ) \cap (\ker \d^*)$.

$\CH \d = \CH \d^* = 0$: Note that $(\CH \d \xi,\eta) = (\xi, \d^* \CH \eta) = 0$ by adjointness. This holds for all $\xi,\eta$, so $\CH \d = 0$ and similarly $\CH \d^* = 0$.

$[G,\d]=0$: $G,\d$ vanishes on $\CH_\Delta$ and thus we only need to prove the theorem for $\CH_\Delta^\perp$. Due to orthogonality of decomposition $I = \CH + \Delta G$, any element in $\CH_\Delta^\perp$ is of the form $\Delta \varphi$. Thus we demand that $G \d \Delta \varphi - \d G \Delta \varphi = 0$. Now, $\d \varphi = (H+G\Delta) \d \varphi = 0 + G \d \Delta \varphi$ and $\d \varphi = \d (H+G\Delta) \varphi = 0 + \d G \Delta \varphi $. Comparing these two expressions give the result.

Finally, isomorphism $H^r \cong \CH^r$ is simply given by $[\varphi] \mapsto \CH(\varphi)$. To show that $\CH(\varphi)=0 \implies \varphi$ is exact form, we use the decomposition $I = \CH + (\d \d^* + \d^* \d) \circ G$ and commutation relation $[\d,G]=0$: If $\CH \xi = 0$, then $\xi = \CH \xi + \d \d^* G \xi + \d^* \d G \xi = 0 + \d \d^* G \xi + \d^* G \d \xi  = 0 + \d (\d^* G \xi) + 0$ and thus $\xi$ is exact.
\end{proof}

This result allows us to ``solve'' Laplace equation on compact complex manifold:
\begin{corollary}
The equation
\begin{align*}
\eta = \Delta \varphi
\end{align*}
has a solution $\varphi$ iff $\mathcal H \eta =0$. The unique solution satisfying $\mathcal H \varphi = 0$ is then given by
\begin{align*}
\varphi = G\eta 
\end{align*}
\end{corollary}
\begin{proof}
If $\eta = \Delta \varphi$, then $\mathcal H \eta = \mathcal H \Delta \varphi = (\Delta - \Delta G \Delta)\varphi = (\Delta - \Delta \Delta G) = \Delta (\mathcal H \varphi) = 0$ since $\mathcal H \varphi \in \mathcal H^r$. Thus $\mathcal H \eta = 0$. Now $G\eta = G \Delta \varphi = (I - \mathcal H)\varphi = \varphi$ as desired, and $\mathcal H \varphi = \mathcal H G \eta = (I - \Delta G)G \eta $.
\end{proof}

\subsection{The Hodge decomposition theorem}

The Hodge decomposition refers to the following isomorphism that is true for a well behaved a class of manifolds: the compact K\"ahler manifolds:
\begin{align*}
& H^r(X,\CC) \rightarrow \bigoplus_{p+q=r}H^{p,q}(X) \\
& [\varphi] \mapsto ([\varphi^{r,0}], [\varphi^{r-1,1}], \cdots, [\varphi^{0,r}])
\end{align*}
i.e. we simply take the $(p,q)$-components. 

Hodge decomposition's proof is basically summarized into $H\cong \mathcal H$ and $\Delta = 2\square$. The former identity, proven in the previous section, relies on functional analysis. The latter identity uses K\"ahler condition and some representation theory of $\mathfrak{sl}(2,\mathbb C)$. Representation theory makes an appearance because commutator (Lie bracket) relations between the operators $L^*,L, B$ (where $B=[L^*,L]$) are exactly the same as the three generators of $\mathfrak{sl}(2,\mathbb C)$

In the proof, adjoints and inner products are heavily used, and they are defined using integration on the manifold. Thus the compactness assumption plays an important role.

\subsubsection{Inner Product on Differential Forms}

We first introduce the notion of Hodge dual (or Hodge star), an operator that takes a differential form to its `complement'. For now, we work in a purely algebraic setting. Let $V$ be a real vector space of dimension $d$ with an inner product. Suppose $\{e_1, \cdots e_d\}$ is a basis for $V$. In the exterior algebra $\wedge V$, endow each $\wedge^p V$ with inner product by declaring $\{e_{i_1} \wedge \cdots \wedge e_{i_p}\}$ to be an orthonormal basis. An orientation on $V$ is a choice of volume form $\on{vol} = e_1 \wedge \cdots \wedge e_d \in \wedge^d V$ for a basis $\{ e_1, \cdots e_d\}$. Now we define the \emph{Hodge dual} operator to be the operator:
\begin{align*}
    *: &\wedge^p V \rightarrow \wedge^{n-p}V \\
    & e_{i_1} \wedge \cdots \wedge e_{i_p} \mapsto \on{sign}(\sigma) e_{j_1} \wedge \cdots \wedge e_{j_{d-p}} \\
    & \text{where } \sigma = (i_1, \cdots i_p, j_1, \cdots j_{d-p}) \in S_d
\end{align*}
Note that we have
$$\alpha \wedge *\beta = \langle \alpha, \beta \rangle \on{vol}$$
for $p$-forms $\alpha, \beta$.

With a fixed inner product, definition of the Hodge star can be shown to be independent of the choice of an orthonormal basis (but flips sign with rearrangement of the orthonormal basis).

Let $F$ be a Hermitian vector space of dimension $n$ and let $F_\RR$ be its underlying real vector space of dimension $2n$. Let $E = (F_\RR)^* \otimes \CC$ be the complex-valued functionals (Later, tangent space takes the role of $F$). Then each functional in $E$ can be decomposed into complex-linear and conjugate-linear parts: $F \cong F^{(1,0)} \oplus F^{(0,1)}$. Conjugation is defined in $F$ by $\bar\omega : (v_1, \cdots v_n) \mapsto \overline{\omega(v_1, \cdots v_n)}$. Given a basis $\{ z_1, \cdots z_n \}$ of $F^{(1,0)}$, we can get a basis for $F^{(0,1}$ by conjugation: $\{\bar z_1, \cdots \bar z_n \}$. Declaring $\wedge^{p,q}F$ to be the subspace of $\wedge F$ generated by $p$ elements of $\wedge^{(1,0)}F$ and $q$ elements of $\wedge^{(0,1)}$, we obtain a bigrading decomposition:
\begin{align*}
    \wedge F = \sum_{r=0}^{2n} \sum_{p+q=r} \wedge^{p,q}F
\end{align*}

Let $\Pi_r: \wedge F \rightarrow \wedge^r F, \Pi_{p,q}: \wedge F \rightarrow \wedge^{p,q} F$ be the projections to $r$-component and $(p,q)$-components. We also define the following sign modification operators:
\begin{align*}
    & w, J: \wedge F \rightarrow \wedge F \\
    & w = \sum_{r=1}^{2n} (-1)^r \Pi_r \\
    & J = \sum_{p,q=1}^{2n} i^{p-q} \Pi_{p,q}
\end{align*}

Using Hermitian metric on $E$, we get the fundamental form $\Omega$:
\begin{align*}
    h = \sum_{\mu,\nu} h_{\mu \nu} z_\mu \otimes \bar z_\nu \implies \Omega = \frac i2 \sum_{\mu,\nu} h_{\mu \nu} z_\mu \wedge \bar z_\nu
\end{align*}
which gives rise to the operator 
\begin{align*}
    & L : \wedge^{p,q} E \rightarrow \wedge^{p+1,q+1} E \\
    & L=\Omega \wedge
\end{align*}
With respect to the inner product of $F$ extended to $\wedge E$, we can define adjoint of $L$.

\begin{proposition}
    The following are true:
    \begin{align*}
        & L^* = w * L * \\
        & * \Pi_{p,q} = \Pi_{n-q,n-p} * \\
        & [L,w]=[L,J]=[L^*,w]=[L^*,J]=0
    \end{align*}
\end{proposition}

We now claim that $[L^*, L] = \sum_{p=0}^{2n} (n-p)\Pi_p$. This interesting identity requires some work to establish, although the work is mostly about rearranging things.

Since $\{x_1, y_1, \cdots, x_n, y_n\}$ is our orthonormal basis, it is not necessarily straightforward to compute the Hodge dual of a $r$-form written in terms of another (not orthogonal) basis $\{z_\mu, \bar z_\nu\}$. However, with right notation, the computation is easy enough. 

Note first that each $z_I\wedge \bar z_K$ can also be written as $z_A \wedge \bar z_B \wedge w_M$ where $A,B,M$ are mutually disjoint increasing multi-indices, and $w_M = (z_{\mu_1}\wedge \bar z_{\mu_1}) \wedge \cdots \wedge (z_{\mu_m}\wedge \bar z_{\mu_m})$. Here, we can obtain $A,B,M$ by $M=I\cap K$, $A=I-I\cap K$, $B=K-I\cap K$. The two notations translate from one to another by the following:
\begin{lemma}
\begin{align*}
z_I\wedge \bar z_K = (-1)^{m(p+q)+\lfloor \frac m 2\rfloor + \tau_{I,K}} (z_{A}\wedge z_{B} \wedge w_{M})
\end{align*}
with $A=I-I\cap K, B=K-I\cap K, M=I\cap K$, and $\tau_{I,K}=\sum_{\mu \in I\cap K} \tau_{\mu,I}+\tau_{\mu,K}$ and $\tau_{\mu,I}$=smallness ranking of $\mu$ in $I$.
\end{lemma}

We now have the following identity:

\begin{lemma}
\begin{align*}
*(z_A\wedge \bar z_B \wedge z_M) = \gamma \cdot (z_A \wedge \bar z_B \wedge z_{N-(A\cup B \cup M)})
\end{align*}
where $\gamma=i^{a-b}(-1)^{p(p+1)/2+m}(-2i)^{p-n}$ with $p=a+b+2m$ the total degree of our form.
\end{lemma}
\begin{proof}
Let's first note that
\begin{align*}
w_M &= z_{\mu_1} \wedge \bar z_{\mu_1} \wedge \cdots \wedge z_{\mu_m} \wedge \bar z_{\mu_m}
\\ &=(x_{\mu_1} + iy_{\mu_1})\wedge (x_{\mu_1} - iy_{\mu_1})\wedge \cdots \wedge (x_{\mu_m} + iy_{\mu_m})\wedge (x_{\mu_m} - iy_{\mu_m})
\\ &= (-2i)^m (x_{\mu_1} \wedge y_{\mu_1} \wedge \cdots \wedge x_{\mu_m} \wedge y_{\mu_m}
\end{align*}
with this, the power of 2 disappears completely and our claim becomes equivalent to:
\begin{align}
\label{hodge1}
& *(z_A \wedge \bar z_B \wedge (x_{\mu_1} \wedge y_{\mu_1} \wedge \cdots \wedge x_{\mu_m} \wedge y_{\mu_m}) )
\\ = & (-2i)^{(n-a-b-m)-m}i^{a-b}(-1)^{p(p+1)/2+m}(-2i)^{p-n} (z_A \wedge \bar z_B \wedge (x_{\nu_1} \wedge y_{\nu_1} \wedge \cdots \wedge x_{\nu_m'} \wedge y_{\nu_m'}) ) \nonumber
\\ = & i^{a-b}(-1)^{\frac12 p(p+1)+m} (z_A \wedge \bar z_B \wedge (x_{\nu_1} \wedge y_{\nu_1} \wedge \cdots \wedge x_{\nu_m'} \wedge y_{\nu_m'}) ) \nonumber
\end{align}
where $\nu$ is complementary index to $\mu$ with $m'=n-a-b$.
Given that $z_A\wedge \bar z_B$ expands into a long list of $(a+b+m)$-forms again, it will be convenient to treat each summand separately. To write this down with generality, let's introduce a index notation to control $x$ and $y$: $x_{i,0}=x_i, x_{i,1}=y_i$. Also, let $\xi: \{1, \cdots , a\} \rightarrow \{0,1\}, \zeta: \{1, \cdots, b\} \rightarrow \{0,1\}$. Then
\begin{align*}
z_A\wedge \bar z_B &= \sum_{\xi,\zeta} (i^{\#\xi^{-1}(1)} x_{1,\xi(1)} \wedge \cdots \wedge x_{a,\xi(a)}) \wedge ((-i)^{\#\zeta^{-1}(1)} x_{1,\zeta(1)} \wedge \cdots \wedge x_{b,\zeta(b)})
\\&=\sum_{\xi,\zeta} i^{\#\xi^{-1}(1) -\#\zeta^{-1}(1)} (x_{1,\xi(1)} \wedge \cdots \wedge x_{a,\xi(a)}) \wedge (x_{1,\zeta(1)} \wedge \cdots \wedge x_{b,\zeta(b)})
\end{align*}
Thus, in equation~(\ref{hodge1}), the dual for a term corresponding to $(\xi,\zeta)$ from expansion of left hand side would be that corresponding to $(1-\xi,1-\zeta)$ on the right hand side. The terms therefore correspond bijectively with each choice of $(\xi,\zeta)$, and all we need to prove is that their signs match up. Since applying Hodge dual tags a signature, the only thing left to prove is the following purely combinatorial identity:
\begin{align}
sign(\Xi) &= i^{-\# \xi^{-1}(1) + \# \zeta^{-1}(1)}i^{\# \xi^{-1}(0) - \# \zeta^{-1}(0)}i^{a-b}(-1)^{\frac12 p(p+1)+m} \nonumber
\\ &= i^{-(a-b)+2(\#\xi^{-1}(0) - \#\zeta^{-1}(0))+(a-b)+p(p+1)+2m} \nonumber
\\ &= (-1)^{\#\xi^{-1}(0) + \#\zeta^{-1}(0) + \frac12 p(p+1)+m} \nonumber
\\ &= (-1)^{\#\xi^{-1}(0) + \#\zeta^{-1}(0) + \frac12 (a+b+2m)(a+b+2m+1)+m} \nonumber
\\ &= (-1)^{\#\xi^{-1}(0) + \#\zeta^{-1}(0) + \frac12 (a+b)(a+b+1)+m(2a+2b+1)+2m^2+m} \nonumber
\\ &= (-1)^{\#\xi^{-1}(0) + \#\zeta^{-1}(0) + \frac12 (a+b)(a+b+1)} \label{hodge2}
\end{align}
where $\Xi$ is the long multi-index 
\begin{align*}
\Xi =(x_{1,\xi(1)} \cdots x_{a,\xi(a)}, x_{1,\zeta(1)}, \cdots, x_{b,\zeta(b)}, x_{\mu_1}, y_{\mu_1}, \cdots, x_{\mu_m}, y_{\mu_m}, C)
\end{align*}
with $C$ being the complement to the preceding multi-index.

The signage is the signage of permutation that takes $(x_1, y_1, \cdots, x_n, y_n)$ to $\Xi$. Since $sign(x_{\sigma_1}, y_{\sigma_1}, \cdots, x_{\sigma_n}, y_{\sigma_n})=1$ for $\sigma \in S_n$, $sign(\Xi)$ is equal to the signage of permutation that takes
\begin{align*}
(x_{\alpha_1}, \cdots y_{\alpha_a}, x_{\beta_1}, \cdots y_{\beta_b}, x_{\mu_1} \cdots y_{\mu_m}, x_{\nu_1} \cdots y_{\nu_{m'}}x_{\nu_1} \cdots y_{\nu_{m'}})
\end{align*}
to $\Xi$. Now it remains to move some indices to the left. 
\begin{align*}
&\text{Moving }x_{1,\xi(1)}, \cdots, x_{a,\xi(a)} = (0+\xi(1))+(1+\xi(2))+\cdots + ((a-1)+\xi(a)) \text{ moves}
\\ &\text{Moving }x_{1,\zeta(1)}, \cdots, x_{b,\zeta(b)} = (a+0+\zeta(1))+\cdots + (a+ (b-1)+\zeta(b)) \text{ moves}
\\ &\text{Moving }x_{\mu_1}, \cdots, y_{\mu_m} = 2m(a+b) \text{ moves}
\end{align*}
Thus,
\begin{align*}
\text{Total} &= 2m(a+b)+ab+\sum_t \xi(t)+\zeta(t) + \frac12 a(a-1) + \frac 12 b(b-1)
\\ & \equiv 0 + \frac12 a(a-1) + \frac 12 b(b-1) + \# \xi^{-1} (1) + \# \zeta^{-1} (1)
\\ & \equiv ab + \frac12 a(a+1) + \frac 12 b(b+1) + \# \xi^{-1} (0) + \# \zeta^{-1} (0)
\\ & \equiv \frac12 (a+b)(a+b+1) + \# \xi^{-1} (0) + \# \zeta^{-1} (0)
\end{align*}
as desired earlier in equation~(\ref{hodge2})
\end{proof}
In this statement, $\gamma$ has all the $i$'s and $2$'s because $w_M = \bigwedge (x_{\mu_j} + i y_{\mu_j})\wedge (x_{\mu_j} - i y_{\mu_j}) = (-2i)^m \bigwedge x_{\mu_j} \wedge y_{\mu_j}$ and the $i$'s coming from the expansion of $z_A, \bar z_B$. With these $2$ and $i$ taken care of, the identity becomes a purely combinatorial one.

Based on this Hodge dual calculation, we can carry out explicit calculation for $L^*=w*L*$:
\begin{proposition}
\begin{align*}
L(z_A\wedge\bar z_B\wedge w_M) &= \frac i2 \sum_{\mu \notin M} z_A \wedge \bar z_B \wedge w_{M\cup\{\mu\}}
\\L^*(z_A\wedge\bar z_B\wedge w_M) & = \frac 2i \sum_{\mu \in M} z_A \wedge \bar z_B \wedge w_{M-\{\mu\}}
\end{align*}
\end{proposition}
\begin{proof}
\begin{align*}
L^*(z_A\wedge\bar z_B\wedge w_M) &= w*L*(z_A\wedge\bar z_B\wedge w_M)
\\ &=\frac i2 (-1)^p \gamma(a,b,m) *\left( (\sum_\mu z_\mu \wedge \bar z_\mu)\wedge (z_A \wedge \bar z_B \wedge w_{M'}) \right)
\\ &=\frac i2 (-1)^p \gamma(a,b,m) \sum_{\mu \in M} * (z_A \wedge \bar z_B \wedge w_{\mu \vee M'} )
\\ &=\frac i2 (-1)^p \gamma(a,b,m) \gamma(a,b,m'+1) \sum_{\mu \in M} z_A \wedge \bar z_B \wedge w_{M-\{\mu\}}
\\ &=\frac i2 (-1)^{p+(a-b)+\binom p2 + \binom{p'+2}2 + m+ (m'+1) }(-2i)^{p-n + (p'+2) -n} \sum_{\mu \in M} z_A \wedge \bar z_B \wedge w_{M-\{\mu\}}
\\ &= (-2i) \sum_{\mu \in M} z_A \wedge \bar z_B \wedge w_{M-\{\mu\}}
\end{align*}
the last equality is true since:
\begin{align*}
(-2i)^{p-n+p'+2-n} &= (-2i)^{a+b+2m-2n + 2 + 2n - 2m - a-b}
\\ &=(-2i)^2
\\ &=-4
\end{align*}
and
\begin{align*}
& p+(a-b)+\binom p2 + \binom{p'+2}2 + m+ (m'+1)
\\ \equiv & p+(a-b) +\binom p2 + \binom{p'}2 + m+ m'
\\ =& (a-b) + \frac12 p(p+1) + \frac12 p'(p'-1) + m + m'
\\ \equiv & (a-b) + \frac12 (a+b+2m)(a+b+1+2m) 
\\ & + \frac12 (2n-2m-a-b)(2n-2m-a-b-1) + m + n-a-b-m
\end{align*}
\begin{align*}
\equiv & (a-b) + \binom{a+b+1}2 + 2m^2 + m(2a+2b+1) + 2(n-m)^2 
\\ & + \binom{a+b+1}2 + (n-m)(-2a-2b-1) + n-a-b
\\ \equiv & n + n(2a+2b+1)
\\ \equiv & 0
\end{align*}
where the first equality is true since $\binom{x}2 \equiv \binom{x+2} 2 - 1 (\text{mod }2)$.
\end{proof}

From this, we can prove the following commutation relation:
\begin{proposition}
\begin{align*}
[L^*,L] = \sum_{p=0}^{2n} (n-p)\Pi_p
\end{align*}
\end{proposition}
\begin{proof}
The computation goes:
\begin{align*}
L^*L (z_A\wedge \bar z_B \wedge w_M) &=z_A \wedge \bar z_B \wedge \left(\sum_{\nu \notin M, \mu \in M} w_{M-\{\mu \}\cup \{\nu\}} + (n-a-b-m) w_M \right)
\\ LL^*(z_A\wedge \bar z_B \wedge w_M) &=z_A \wedge \bar z_B \wedge \left(\sum_{\nu \notin M, \mu \in M} w_{M-\{\mu \}\cup \{\nu\}} + m w_M \right)
\end{align*}
Here, $L^*L$ has the $(n-a-b-m)$ since $L^*$ is ``erasing'' from variables added by $L$, with $L$ being able to place indices at $(n-a-b-m)$ places. $LL^*$ has the $m$ since $L$ is filling variables erased by $L^*$, with $L^*$ erasing indices at $m$ places. Subtracting, we obtain
\begin{align*}
[L^*,L](z_A\wedge \bar z_B \wedge w_M) &= z_A\wedge \bar z_B \wedge (n-a-b-2m)w_M
\\ &= (n-p)(z_A\wedge \bar z_B \wedge w_M)
\end{align*}
and thus the proof is complete.
\end{proof}

Let $X$ be a compact complex manifold. Then we may define an inner product on $\mc E^*(X)$ by setting for $\xi \in \mc E^p(X), \eta \in \mc E^q(X)$
\begin{align*}
    (\xi, \eta) = \begin{cases} \int_X \xi \wedge * \bar \eta & \text{ when $p=q$} \\ 0 & \text{ when $p \neq q$} \end{cases}
\end{align*}
This inner product is called the Hodge inner product, and induces the following orthogonal decomposition:
\begin{align*}
    \mc E^r(X) \cong \bigoplus_{p+q=r} \mc E^{p,q}(X)
\end{align*}

Let's compute the adjoint of exterior derivative $d$ with respect to the Hodge inner product.
\begin{proposition}
    $d^* = (-1)^{n(m+1)+1}*d*$
\end{proposition}

\begin{proposition}
    For differential form $\varphi=fdx_J$ ($J$ is multi-index $J=(j_1, \cdots, j_m)$) and orthonormal coordinates $x=(x_1, \cdots x_n)$, 
    \begin{align*}
        d^*(f dx_J) = \sum_{p} (-1)^{p}\frac{\partial f}{\partial x_{j_p}}dx_{J\backslash j_p}
    \end{align*}
\end{proposition}
\begin{proof}
Step-by-step, 
    \begin{align*}
    *(f dx_J) &= \text{sgn}(J, \bar J)f dx_{\bar J} \\
    d*(fdx_J) &= \text{sgn}(J,\bar J) \sum_p \frac{\partial f}{\partial x_{j_p}}(dx_{j_p} \wedge dx_J) \\
    *d*(fdx_J) &= \text{sgn}(J,\bar J) \sum_p \text{sgn}(j_p \vee \bar J,J\backslash j_p) \frac{\partial f}{\partial x_{j_p}}(dx_{J\backslash j_p}) \\
    &= \sum_p \text{sgn}(J,\bar J) (-1)^{n-m+p+1}\text{sgn}(\bar J,J) \frac{\partial f}{\partial x_{j_p}}(dx_{J\backslash j_p}) \\
    &= \sum_p (-1)^{n-m+p+1}(-1)^{(n-m)m} \frac{\partial f}{\partial x_{j_p}}(dx_{J\backslash j_p}) \\
    d^*(fdx_J) &= \sum_p (-1)^{n(m+1)+1+n-m+p+1+(n-m)m} \frac{\partial f}{\partial x_{j_p}}(dx_{J\backslash j_p}) \\
    &= \sum_p (-1)^p \frac{\partial f}{\partial x_{j_p}}(dx_{J\backslash j_p})
    \end{align*}
\end{proof}

We can now give a computation for Laplacian $\Delta = dd^*+d^*d$.
\begin{proposition}
    In orthonormal coordinates $x=(x_1, \cdots x_n)$ on a Riemannian manifold, the Laplacian $\Delta = dd^*+d^*d$ has expression
    \begin{align*}
        \Delta = \sum_k \frac{\partial^2}{\partial x_k^2}
    \end{align*}
\end{proposition}
\begin{proof}
Suppose $\varphi = f dx_J$ where $J=(j_1, \cdots, j_m)$ a multi-index. Let's compute $dd^*\varphi$ and $d^*d\varphi$ separately.
\begin{align*}
    dd^*(fdx_J) &= d(\sum_{p} (-1)^{p}\frac{\partial f}{\partial x_{j_p}}dx_{J\backslash j_p})
    \\ &= \sum_p (-1)^p \left(\frac{\partial^2 f}{\partial x_{j_p}^2} dx_{j_p} \wedge dx_{J\backslash j_p} + \sum_{	k\notin J} \frac{\partial^2 f}{\partial x_{k} \partial x_{j_p}} dx_{k} \wedge dx_{J\backslash j_p} \right)
    \\ &= \left(\sum_p \frac{\partial^2 f}{\partial x_{j_p}^2} dx_J\right) + \left( \sum_p \sum_{k\notin J} (-1)^p \frac{\partial^2 f}{\partial x_{k} \partial x_{j_p}} dx_{k} \wedge dx_{J\backslash j_p} \right) \\
    d^*d(fdx_J) &= d^*(\sum_{k\notin J} \frac{\partial f}{\partial x_k} dx_k \wedge dx_J)
    \\ &= \left(\sum_{k\notin J} \frac{\partial^2 f}{\partial x_k^2}\right) + \left(\sum_p \sum_{k\notin J} (-1)^{p+1}\frac{\partial^2 f}{\partial x_{j_p} \partial x_k}dx_k \wedge dx_{J\backslash j_p} \right) \\
    \Delta (fdx_J) &= \sum_{k} \frac{\partial^2 f}{\partial x_k^2} dx_J
\end{align*}
\end{proof}

The following is an explicit calculation of $d^*$:
\begin{proposition}
    \begin{align*}
        d^* (f dz_A \wedge d\bar z_B \wedge w_M) &= 2(-1)^{m(p+1)+(a+b)} \left( \sum_{\mu \in M} (\frac{\partial f}{\partial z_\mu} + \frac{\partial f}{\partial \bar z_\mu}) \wedge (z_A \wedge \bar z_B \wedge w_{M-\{\mu\}} )\right)
        \\&= i \cdot df \wedge L^*(f dz_A\wedge d\bar z_B \wedge w_M)
    \end{align*}
and thus
    \begin{align*}
        d^* = (-1)^{1+m(p+1)}*d* = i \cdot df \wedge L^*
    \end{align*}
\end{proposition}

\subsubsection{Representation Theory of $\mathfrak{sl}_2\mathbb C$}

Let $B=\sum_{r=0}^{2n} (n-p)\Pi_r$. Then, from above, the following relations hold:
\begin{align}
[B,L^*]=2L^*, [B,L]=-2L, [L^*,L]=B
\end{align}
This is actually completely analogous to:
\begin{align}
[H,X]=2X, [H,Y]=-2Y, [X,Y]=H
\end{align}
where
\begin{align}
X=
\begin{bmatrix}
0&1\\0&0
\end{bmatrix}
, Y=
\begin{bmatrix}
0&0\\1&0
\end{bmatrix}
, H=
\begin{bmatrix}
1&0\\0&-1
\end{bmatrix}
\end{align}
Moreover, $H,X,Y$ are generators of the Lie algebra $\mathfrak{sl}_2\mathbb C$ (matrices with trace 0). This means that we have a Lie algebra representation of $\mathfrak{sl}_2\mathbb C$ on $\wedge F$:
\begin{align}
\rho &: \mathfrak{sl}_2\mathbb C \rightarrow End(\wedge F)
\\ & X \mapsto L^*, Y \mapsto L, H \mapsto B
\end{align}
Thus, studying representation theory of $\mathfrak{sl}_2\mathbb C$ will help us study the operators $L,L^*$.

The standard picture is to think of a ``Ladder" with $L^*$ raising by a rung and $L$ dropping by a rung.  This gives the unique picture of what an irreducible $\mf{sl}(2,\CC)$-representation looks like.  Then we can decompose any representation (in our case, that on exterior Hermitian algebra $\wedge T^*X$) into irreducible representations. This is a result that relies on considerations of integrations on Lie groups, which is quite interesting considering that $\mf{sl}(2,\CC)$ admits a purely algebraic description independent of any manifold formalism.

\subsubsection{The Hodge decomposition for K\"ahler manifolds}

Here we discuss a special well-behaved class of complex manifolds called K\"ahler manifolds, and prove the Hodge decomposition theorem which hold for K\"ahler manifolds.

\begin{definition}
    A Hermitian metric $h$ on complex manifold $X$ is said to be \emph{K\"ahler} if the fundamental form associated to it is a closed form, i.e. $\d \Omega = 0$.
\end{definition}

\begin{proposition}
    A Hermitian metric $h$ on a complex manifold $X$ is K\"ahler iff any of the following holds:
    \begin{enumerate}
        \item The fundamental form induced by $h$ is closed: $d\Omega=0$.
        \item $[d,L]=0$ where $L:\omega \mapsto \Omega \wedge \omega$.
        \item The metric admits a local normal coordinate (is orthonormal and metric coefficients' first derivatives vanish).
        \item The metric coefficients satisfy the following:
        \begin{align*}
        \frac{\partial h_{\mu\nu}}{\partial z_\lambda} &= \frac{\partial h_{\lambda\nu}}{\partial z_\mu}
        \\ \frac{\partial h_{\mu\nu}}{\partial z_\lambda} &= \frac{\partial h_{\mu\lambda}}{\partial z_\nu}
        \end{align*}
    \end{enumerate}
\end{proposition}

\begin{proof}
    Suppose $L$ is well-defined on de Rham cohomology ring. Then every exact form, wedged with $\Omega$, must be exact too:
    \begin{align*}
        \forall \eta, \exists \varphi: \Omega \wedge \text d \eta = \text d \varphi
    \end{align*}
    Taking $\text d$ of both sides and using $\Omega \wedge \text d \eta = \text d (\Omega \wedge \eta) - (\text d \Omega \wedge \eta)$, we obtain:
    \begin{align*}
        \forall \eta, \exists \varphi: \text d \Omega \wedge \text d \eta = 0
    \end{align*}
    We are to establish $\text d \Omega = 0$. If we were working over $\mathbb C^n$, then we can simply plug in $\eta = \text dz_j,\text d\bar z_j$ for different $j$'s and we immediately obtain that all coefficients of $d\Omega$ are zero. For general compact Hermitian manifold $M$, we can use bump functions to show that the form is locally 0 and thus globally 0. Take an open cover of coordinate patches of open balls $\{B_j\}$. Take balls $\{B_j'\}$ of radius half of those initially chosen. Construct smooth bump functions $\phi_j:M \rightarrow \RR$ such that $\phi_j$ vanishes outside $B_j$ and is identically $1$ at $B_j'$. Then plug in $\eta = \phi_k \text dz_j, \phi_k \text d\bar z_j$ for varying $j,k$ to obtain that local expression of $d\Omega$ must be identically 0 at any patch.
\end{proof}

With K\"ahler assumption, we have $[\d, L]$, and its adjoint statement is $[\d^*, L^*]=0$ ($\forall \xi, \eta, ([\d^*,L^*] \xi, \eta) = (\xi, [L, \d]\eta) = 0 \implies [d^*,L^*]=0$).

Define twisted conjugate of an operator $P$: $P_c := J^{-1} P J$. We then see that 
\begin{align*}
    d_c =& J^{-1} \d J = wJ \d J \\
    \d_c \varphi =& w J \d J \varphi \\
    =& (-1)J(\dol \varphi + \bar\dol \varphi) \\
    =& (-1)(i \dol \varphi - i \bar\dol \varphi) \\
    =& -i(\dol - \bar\dol) \varphi
\end{align*}
and similarly $d^*_c = -i(\dol^* - \bar\dol^*)$ (these can also be taken as definitions of the twisted conjugates).

The following nontrivial identities are true for compact K\"ahler manifolds, and can be derived using the representation theory of $\mf{sl}(2,\CC)$, using exponential map.
\begin{proposition}
    On a compact K\"ahler manifold $X$,
    \begin{align*}
        [L, d^*] = d_c, [L^*, d] = - d_c^*
    \end{align*}
\end{proposition}

The following are corollaries of the identity:
\begin{corollary}
    On a compact K\"ahler manifold $X$,
    \begin{align*}
        & [L, \d_c] = [L^*, \d_c^*] = 0 \\
        & [L, \d_c^*] = -d, [L^*, \d_c] = \d^* \\
        & [L,\dol] = [L, \bar\dol] = [L^*, \dol^*] = [L^*, \bar\dol^*] = 0\\
        & [L, \dol^*] = i\bar\dol, [L,\bar\dol^*] = -i\dol \\
        & [L^*, \dol] = i\bar\dol^*, [L^*, \bar\dol] = -i\dol^* \\
        & \d^* \d_c = -\d_c \d^* = \d^* L \d^* = -\d_c L^* \d_c \\
        & \d \d_c^* = - \d_c^* \d = \d_c^* L \d_c^*  = -\d L^* \d \\
        & \dol \bar\dol^* = -\bar\dol^* \dol = -i \bar\dol^* L \bar\dol^* = -i\dol L^* \dol \\
        & \bar\dol \dol^* = - \dol^* \bar\dol = i\dol^* L \dol^* = i \bar\dol L^* \bar\dol
    \end{align*}
\end{corollary}

\begin{proposition}
    $\Delta_c = \Delta$
\end{proposition}
\begin{proof}
    From $d^* = [L^*, d_c]$, $[L,J]=[L^*,J]=0$, this follows easily. Firstly, write:
    \begin{align*}
        \Delta 	& = dd^* + d^* d
        \\ 			& = d[L^*,d_c] + [L^*,d_c]d
        \\			& = dL^* d_c - d d_c L^* + L^* d_c d - d_c L^* d
    \end{align*}
    and now, applying $J^{-1}$ on the left and $J$ on the right to each summand,
    \begin{align*}
        J^{-1} (dL^* d_c) J &= J^{-1} (d J J^{-1} L^* (-J d J^{-1})) J
        \\					&= - (J^{-1} d J) (J^{-1} L^* J) d
        \\					&= - d_c (J^{-1} J L^*) d
        \\					&= - d_c L^* d
    \end{align*}
    and working similarly for other summands, we obtain:
    \begin{align*}
    \Delta_c &= J^{-1} \Delta J
        \\			&= -d_c L^* d + d_c d L^* - L^* d d_c + d L^* d_c
        \\			&= dL^* d_c + d_c d L^* - L^* d d_c - d_c L^* d
        \\			&= dL^* d_c - d d_c L^* + L^* d_c d - d_c L^* d
        \\			&= \Delta
    \end{align*}
    where in last equality we used $\{d,d_c\}=0$, which can be proven via direct computation on a $(p,q)$-form just like the previous proof of $2\dol = -i(\dol - \bar \dol)$.
\end{proof}

\begin{proposition}
    $\Delta = 2 \square = 2\bar\square$
\end{proposition}
\begin{proof}
    \begin{align*}
        4\square &= 4(\dol \dol^* + \dol^* \dol)
        \\ &= (d+id_c)(d^* - id_c^*) + (d^* - id_c^* ) (d+id_c)
        \\ &= (dd^* + d^*d) + (d_cd_c^* + d_c^* d_c) +i (d_c d^* + d^* d_c) - i (dd^*_c + d^*_c d)
        \\ &= \Delta + \Delta_c + i\{d_c, d^*\} - i \{ d, d^*_c \}
        \\ &= \Delta + \Delta_c + 0 + 0
        \\ &= 2\Delta
    \end{align*}
\end{proof}

\begin{theorem}[Hodge Decomposition]
    For a compact K\"ahler manifold,
    \begin{align*}
        H^r(X,\CC) \cong \bigoplus_{p+q=r} H^{p,q}(X)
    \end{align*}
\end{theorem}
\begin{proof}
    Noting that $\Delta = 2\bar \square$ and the harmonic theory established earlier, the following diagram will give the proof:
    \[\begin{tikzcd}
    H^r(X,\CC) \ar[r,dashed] \ar[d] & \bigoplus_{p+q=r} H^{p,q}(X) \ar[d] \\
    \mc H^r(X) \ar[r] & \mc \bigoplus_{p+q=r} \mc{H}^{p,q}(X)
    \end{tikzcd}\]
    The map $\mathcal H^r(X) \cong \bigoplus_{p+q=r} \mathcal H^{p,q}(X)$ is given by bidegree splitting of a $r$-form into $(p,q)$-forms: for $\varphi \in \mathcal H^r(X)$, 
    \begin{align*}
        \varphi = \varphi^{0,r}+\varphi^{1,r-1}+\cdots + \varphi^{r-1,1}+\varphi^{0,r}
    \end{align*}
    where $\varphi^{p,q}\in \mathcal H^{p,q}(x)$ give a unique decomposition. Then, since $\Delta \varphi = 0$ by assumption, 
    \begin{align*}
        \Delta(\varphi) &= \Delta (\sum \varphi^{p,q}) 
        \\ &= 2\bar\square(\sum \varphi^{p,q})
        \\ &=\sum \bar\square(\varphi^{p,q})
        \\ &=0
    \end{align*}
    Since $\bar\square$ preserves bidegree, each of the summand $\bar\square(\varphi^{p,q})=0$ too, thus giving a well-defined map $\mathcal H^{r}(X) \rightarrow \bigoplus_{p+q=r} \mathcal H^{p,q}(X)$. 
    
    This map is injective since if bidegree-components of two different $r$-forms agree, then they are the same. This map is surjective since each $\bar\square$-closed $(p,q)$-form is also a $\Delta$-closed $(p+q)=r$-form.
\end{proof}
    
Every smooth complex projective variety admits a K\"ahler metric inherited from the Fubini-Study metric of projective space, and thus they each satisfy Hodge decomposition. Meanwhile, an analogue of Hodge decomposition would not work in general for singular varieties. An example is variety $V(x^3+y^3=xyz)$, for which $\dim H_1(V(x^3+y^3=xyz))=1$ and thus there can't be an analogue of Hodge decomposition. Since Dolbeault groups were defined using differential forms, we attempt to formulate the decomposition immediately. However, a variant called $L^2$ cohomology does the work in some cases.

Here, we derive some interesting corollaries of the decomposition theorem.

\medskip
\textbf{Cohomology Groups of Projective Space: $H^q(\PP^n,\Omega^p)=H^{p,q}(\PP^n)=\mathbb Z^{\delta_{p,q}}$}

In view of Dolbeault theorem, we see that for compact K\"ahler manifolds, sheaf cohomology of holomorphic forms is controlled by topological conditions (de Rham cohomology). For spaces with simple cohomology, such as $\PP^n$, this immediately gives computation of their $H^q(X,\Omega^p)$'s:
\begin{align*}
H^{p,q}(\PP^n) \cong H^q(\PP^n,\Omega^p)=\begin{cases}\ZZ & \text{if }p=q \\ 0 & \text{if }p\neq q \end{cases}
\end{align*}

\medskip
\textbf{Mittag-Leffler for $\PP^n$}

With the above knowledge of $H^{0,1}(\PP^n)\cong H^1(\PP^n,\mathcal O)=0$, we see that Mittag-Leffler problem for projective spaces is immediately solved. Thus, we can find always find a meromorphic function with prescribed principal parts on $\PP^n$.

\medskip
\textbf{Line Bundles of $\PP^n$ are completely determined by Chern class}

Furthermore, $H^1(\PP^n,\mathcal O)=H^2(\PP^n,\mathcal O)=0$ gives exact sequence
\begin{align*}
\textcolor{gray}{H^1(\PP^n,\mathcal O) =} 0\rightarrow H^1(\PP^n,\mathcal O^\times) \rightarrow H^2(\PP^n, \ZZ)\rightarrow 0 \textcolor{gray}{= H^2(\PP^n,\mathcal O)}
\end{align*}
coming from exponential sheaf sequence. Thus $H^1(\PP^n,\mathcal O^\times) \cong H^2(\PP^n,\ZZ) \cong \ZZ$ and line bundles of $\PP^n$ (parametrized by $H^1(\PP^n,\mathcal O^\times)$) correspond to a single invariant, which turns out to be Chern class because of Bockstein morphism in the sequence.

\medskip
\textbf{Cohomology Groups of Compact Stein Manifolds: $H^q(X,\Omega^p) = H^{p,q}(X) = \ZZ^{\delta_{q,0}}$}

We have
\begin{align*}
& \Omega^r(X) \cong H^{r,0}(X) \cong H^r(X,\CC) \\
& H^q(X,\Omega^p) \cong H^{p,q}(X) = 0 \text{ whenever $q>0$}
\end{align*}
for compact Stein manifolds $X$. We obtain this by applying Hodge decomposition (Stein manifolds are K\"ahler as submanifolds of $\CC^n$) and then applying Cartan's Theorem B:
\begin{align*}
& H^q(X,\mathcal F)=0 \text{ for coherent sheaf $\mathcal F$ and $q>0$} \\
\implies & H^q(X,\Omega^p)=0 \text{ for $q>0$}
\end{align*}

\medskip
\textbf{Equivalence of $d,\dol,\bar\dol-$closedness and exactness}

Interpreting the isomorphism literally, the Hodge decomposition is equivalent to the following two statements:
\begin{itemize}
\item A differential form is $d$-exact $\iff$ $\dol$-exact $\iff$ it is $\bar\dol$-exact.
\item A differential form is $d$-closed $\iff$ $\dol$-closed $\iff$ $\bar\dol$-closed.
\end{itemize}
The first statement is equivalent to injectivity of the decomposition, and half of the well-definedness. The second statement is equivalent to surjectivity and other half of the well-definedness. The first statement actually follows from a lemma called $\dol\bar\dol$-lemma, and it is a direct corollary of theory of elliptic operators. The second statement needs more machinery than this. The standard proof establishes $\Delta = 2\bar\square$ using K\"ahler identities, which implies the second statement. Note that $\Delta \varphi = 0 \iff d\varphi=d^*\varphi=0 \iff d\varphi=-*d*\varphi=0 \iff d\varphi = d(*\varphi)=0$ and similarly $\bar\square \varphi = 0 \iff \bar\dol \varphi = \bar\dol(*\varphi)=0$.

\section{Kodaira Embedding Theorem}

The Kodaira embedding theorem states that a complex manifold is a projective variety iff it has a positive line bundle. A map is constructed using global sections of this line bundle, and the fact that the map is indeed an embedding (well-defined, injective and of full rank) is shown using vanishing of certain first cohomology groups. Positivity of the line bundles implies vanishing of the first cohomology groups. This strategy, however, doesn't work directly and a workaround is accomplished using blow-up of the underlying manifold at (one or two) points. 

We introduce the proof in four subsections. In the first subsection we discuss positive line bundles. In the second subsection we prove the Kodaira-Nakano vanishing theorem, which shows vanishing of cohomology to be used later. In the third subsection, we introduce blow-up and its properties; a positive line bundle must be modified appropriately in a blow-up to still be positive. In the fourth subsection we complete the proof of the Kodaira embedding theorem.

\subsection{Positive Line Bundles}

\begin{definition}
    A $(1,1)$-form is said to be \emph{positive} if it can be written in the form
    $$i \sum_{\mu, \nu=1}^n h_{\mu \nu} \d z_\mu \wedge \d \bar{z}_\nu$$
    where the matrix $(h_{\mu \nu})$ is Hermitian and positive definite.
    
    A holomorphic line bundle $L \rightarrow X$ is said to be \emph{positive} if its (first) Chern class has a representative that is a real positive differential form. In other words, the curvature form is cohomologous to a 2-form with positive definite matrix coefficients.
\end{definition}
The reformulation of the definition comes from the fact that $c_1 = [\frac i{2\pi} \Theta]$.

\begin{lemma}\label{kahler from positive}
    If $E \rightarrow X$ is a positive line bundle over a compact complex manifold, then $X$ is a K\"ahler manifold.
\end{lemma}
\begin{proof}
    Let $i\varphi$ be the positive differential form in the Chern class of the positive line bundle; coefficient matrix of $\varphi$ is positive definite and Hermitian. Then positivity and closedness of $\varphi$ gives $X$ a K\"ahler metric with fundamental form equal to $i \varphi$. 
\end{proof}

When a line bundle is positive, we know that a metric gives curvature $\Theta$ such that $i\Theta$ is cohomologous to some $i\omega$ where $\omega$ is positive. The following proposition states that we can just directly assume that $i \Theta$ is positive, for some metric.

\begin{proposition}\label{positive bundle criterion}
    A holomorphic line bundle $E$ over a compact complex manifold $X$ is positive iff there is a metric that gives curvature $\Theta$ such that $i\Theta$ is positive.
\end{proposition}
\begin{proof}
    When there is such a metric, the line bundle is positive by definition.
    
    Conversely suppose there is a metric $h$ such that its curvature $\Theta_0 = \bar \dol \dol \log h$ is cohomologous to another (1,1)-form $\Theta$:
    $$\Theta = \Theta_0 + \d \eta$$
    Positivity of $\Theta$ makes $X$ a K\"ahler manifold; if $\Theta = \sum_{\mu, \nu} \varphi_{\mu \nu} \d z_\mu \wedge \d \bar z_\nu$, the matrix $(\varphi_{\mu \nu})$ is Hermitian and $\Theta$ is closed, so that there is an associated K\"ahler metric $2\sum \varphi_{\mu \nu} \d z_\mu \otimes \d \bar z_\nu$ on compact complex manifold $X$.
    
    We claim that $\d \eta = \bar \dol \dol (2i L^* G \d \eta)$, which will imply that, for a new metric $h=e^{2i L^* G \d \eta} \cdot h_0$ its curvature is given by 
    $$\bar{\dol} \dol \log h = \bar{\dol}\dol \log h_0 + \bar\dol \dol (2i L^* G \d \eta) = \Theta_0 + \d \eta$$
    Now we prove the claim using K\"ahler identities:
    \begin{align*}
        & \eta = \mc{H} \eta + \Delta G \eta \\
        \implies & \d \eta = \d \mc{H} \eta + \d \Delta G \eta = \Delta G \d \eta \text{ (since $\d \mc{H} = 0, [\d ,\Delta] = [\d, G] = 0$)} \\
        \implies & \d \eta = 2 \bar \dol \bar{\dol}^* G \d \eta + 2  \bar{\dol}^* G \bar{\dol} \d \eta \text{ (since $\Delta = 2 \bar \square$ and(?) $[\bar\dol, G] = 0 $)}\\
        \implies & \d \eta = 2 \bar \dol \bar{\dol}^* G \d \eta \text{ (since $\bar \dol \d \eta = \bar \dol (\varphi - \varphi_0) = 0$ by $\dol \bar \dol$-Poincare lemma)} \\
        \implies & \d \eta = 2 \bar \dol (\frac1i(L^* \dol - \dol L^*)) G \d \eta \text{ (since $[L^*, \dol] = i \bar \dol^*$)}\\
        \implies & \d \eta = 2 i \bar \dol \dol L^* G \d \eta \text{ (since $ [\dol, G]=0, \dol \d \eta = 0$)}
    \end{align*}
\end{proof}

\subsection{The Kodaira-Nakano Vanishing Theorem}

\begin{theorem}[Kodaira-Nakano]
    Suppose $X$ is a compact complex manifold.
    \begin{enumerate}
        \item For a holomorphic line bundle $E \rightarrow X$ such that $E \otimes K_X^*$ is positive,
        $$q>0 \implies H^q(X,\mc O(E)) = 0$$
        \item For a negative line bundle $E \rightarrow X$,
        $$p+q < n \implies H^q(X, \Omega^p(E)) = 0 $$
    \end{enumerate}
\end{theorem}
\begin{proof}
    We firstly note that we may apply Lemma~\ref{kahler from positive} in to assume $X$ to be K\"ahler (for negative $E$, $E^*$ is positive).
    
    We first claim that the following holds for each $\xi \in \mc H^{p,q}(X,E)$:
    \begin{align*}
    & i(\Theta L^* \xi, \xi ) \le 0 \\
    & i(L^* \Theta \xi, \xi) \ge 0
    \end{align*}
    \begin{align*}
        & 0 \le (\dol^* \xi, \dol^* \xi) \\
        =& \frac1i ([\bar \dol, L^*] \xi, \dol^* \xi)) \text{ ($[\bar \dol, L^*] = i \dol^*$ is a K\"ahler identity)} \\
        =& -i (\bar\dol L^* \xi, \dol^* \xi)  \text{ ($\bar \dol \xi = 0$ since $\xi$ is harmonic)}\\
        =& -i (L^* \xi, (\bar\dol^* \dol^* + \dol^* \bar \dol^*)\xi) \text{  ($\bar \dol^* \xi = 0$ since $\xi$ is harmonic)} \\
        =& -i ((D' \bar \dol + \bar \dol D')L^* \xi, \xi) \text{ ($(D')^* = \partial^*$)} \\
        =& -i ((\d + \theta)^2 L^* \xi, \xi) \text{ ($D=D'+D''$ and $D''=\bar \dol$)} \\
        =& -i (\Theta L^* \xi, \xi)
    \end{align*}
    similarly, we establish the second inequality.
    
    We prove the second statement. As $-\frac i2 \Theta$ is a fundamental form for a K\"ahler metric, we may let $L = -\frac i2 \Theta \wedge$. Suppose $\xi \in \mc H^{p,q}(X, E)$. Then
    \begin{align*}
        & 0 \le i (\Theta L^* \xi - L^* \Theta \xi, \xi) = ([L,L^*]\xi, \xi) = -(n-p-q)(\xi, \xi) \\
        \implies & 0 \ge (\xi, \xi) \text{ when $p+q < n$} \\
        \implies & H^q(X, \Omega^p(E)) \cong H^{p,q}(X, E) = 0 \text{ when $p+q < n$}
    \end{align*}
    
    The first statement follows by Serre duality.
\end{proof}

Note that an equivalent formulation of Kodaira-Nakano vanishing theorem is that for every positive line bundle $E$ over a compact complex $n$-manifold,
\begin{align*}
& q > 0 \implies H^q(X, \mc O(E \otimes \Omega^n))=0 \\
& p+q < n \implies H^q(X, \mc O(E^* \otimes \Omega^p))=0
\end{align*}

\subsection{Blow-Up on Points}

We will introduce blow-up of a complex manifold; it `enlarges' a point into a hypersurface (divisor) and allows us to work with divisors instead of a point.

Firstly blow-up of a complex polydisk $U \subseteq \CC^n$ centered at origin is defined as
\begin{align*}
& \tilde U \subset U \times \PP^{n-1} \\
& \tilde U := \{ ((z_1, \cdots z_n), [t_1, \cdots t_n]) | \forall \alpha, \beta, z_\alpha t_\beta = z_\beta t_\alpha \}
\end{align*}
where $[t_1, \cdots t_n]$ is homogeneous coordinate for $\PP^{n-1}$. 
Observe that due to the equations, whenever $(z_1, \cdots z_n) \neq 0, [z_1, \cdots z_n] = [t_1, \cdots t_n]$, i.e. the $U$-coordinate is paired with the line pointing at that coordinate. Thus points away from origin map bijectively to $\tilde U$. Meanwhile at origin $(z_1, \cdots z_n)=0$, any $t$-value is allowed. 

Blow-up is a submanifold of $U \times \PP^{n-1}$ with coordinates defined by 
\begin{align*}
\varphi_\alpha: & \tilde U_\alpha \rightarrow \CC^n \text{ where $\tilde U_\alpha = \{(z,t) \in \tilde U | t_\alpha \neq 0 \}$} \\
& ((z_1, \cdots z_n), [t_1, \cdots t_n]) \mapsto (z_\alpha, \frac{t_1}{t_\alpha}, \cdots \widehat{\frac{t_\alpha}{t_\alpha}}, \cdots \frac{t_n}{t_\alpha})
\end{align*}

Now blow-up of a complex manifold $X$ at point $p$, denoted $\tilde X_p$, is defined by firstly finding a coordinate neighborhood $V$ of $p$, homeomorphic to a polydisk $U \subseteq \CC^n$, and gluing $\tilde U$ with $X$.

Denoting $\pi_p: \tilde X_p \rightarrow X$, note that $S_p :=\pi^{-1}(p)$ is a hypersurface (divisor) in $\tilde X$ locally defined by $z_\alpha=0$ in each $\tilde U_\alpha$. Thus it defines a line bundle whose transition function is given by quotients of the local defining functions.

Also, we may blow up on finitely many points by performing blow-up multiple times; a blow-up on a point is a local operation, so this doesn't depend on the order. Denote the blow-up on two distinct points $p,q$ by $\tilde X_{p,q}$ and denote by $\pi_{p,q}: \tilde X_{p,q} \rightarrow X$ the projection.

We prove some key results about line bundles on blow-up.

\begin{proposition}
    \begin{align*}
    [S_p]_{\tilde U_p} = \sigma^* H^*
    \end{align*}
    where $\sigma : \tilde U \rightarrow \PP^{n-1}$ is given by projection to second coordinate.
\end{proposition}
\begin{proof}
    Local defining functions of $S_p$ at $\tilde U$ are $z_\alpha = 0$, so that its transition function is given by 
    $$g_{\alpha \beta} = \frac{z_\alpha}{z_\beta}$$
    which is reciprocal of the transition function for hyperplane bundle.
\end{proof}

\begin{proposition}
    \begin{align*}
        K_{\tilde X_p} = \pi_p^*K_X \otimes [S_p]^{n-1}
    \end{align*}
\end{proposition}
\begin{proof}
    We first examine $K_{\tilde X_p}$ near $p$ at $\tilde U_p$. Using local holomorhpic coordinates $(z_\alpha, \frac{t_1}{t_\alpha}, \cdots \widehat{\frac{t_\alpha}{t_\alpha}}, \cdots \frac{t_n}{t_\alpha})$ at $\tilde U_\alpha$, we get a frame for the canonical bundle:
    \begin{align*}
    f_1 &= (-1)^\alpha \d z_\alpha \wedge \d (\frac{t_1}{t_\alpha}) \wedge \cdots \widehat{\frac{t_\alpha}{t_\alpha}} \cdots \wedge \d (\frac{t_n}{t_\alpha}) \\
    &= (-1)^\alpha \d z_\alpha \wedge \bigwedge_{\beta \neq \alpha} \left( \frac{-z_\beta}{z_\alpha^2} \d z_\alpha + \frac{1}{z_\alpha} \d z_\beta \right) \\
    &= z_\alpha^{1-n} \d z_1 \wedge \cdots \wedge \d z_n
    \end{align*}
    Thus transition function is given by
    \begin{align*}
        g_{\alpha \beta} = \frac{f_\beta}{f_\alpha} = \left(\frac{z_\beta}{z_\alpha}\right)^{1-n}
    \end{align*}
    which implies
    $$K_{\tilde X_p}|_{\tilde U_p} = [S_p]^{n-1}|_{\tilde U_p}$$
    Meanwhile $\pi_p|_{\tilde X_p \backslash \tilde U_p}$ is isomorphism, so that 
    $$K_{\tilde X_p}|_{\tilde X_p \backslash \tilde U_p} \cong K_X $$
    
    As $K_X|_U$ and $[S_p]|_{\tilde X_p \backslash \tilde U_p}$ are trivial, we see that
    $$K_{\tilde X_p} \cong \pi_p^* K_X \otimes [S_p]^{n-1}$$
    as desired.
\end{proof}

\begin{proposition}\label{blowup positive}
    Given a positive line bundle $E$ over a complex manifold $X$, there is a positive integer $\mu_0$ such that $\forall \mu \ge \mu_0$ and for all distinct $p, q \in X$,
    \begin{align*}
        & \pi_p^*E^\mu \otimes [S_p]^* \otimes K_{\tilde X_p}^* \\
        & \pi_p^*E^\mu \otimes ([S_p]^*)^2 \otimes K_{\tilde X_p}^* \\
        & \pi_{p,q}^*E^\mu \otimes [S_{p,q}]^* \otimes K_{\tilde X_{p,q}}^*
    \end{align*}
\end{proposition}
\begin{proof}
    For convenience, we drop the $p$ subscript here when the point at which blow-up is done is clear.
    
    We will control curvature using the fact that $[S]^*$ is locally a pullback of universal bundle and that away from $S$, $\tilde X$ is the same as $X$ so that we may use positivity of $E$. 
    
    We first endow a metric to $[S]$ so that its curvature is $\Theta_{\sigma^* H}$ near $p$ and zero everywhere else. A metric on hypersection bundle of $\PP^{n-1}$ is given by 
    $$\Theta_H = \bar \dol \dol \log \frac{|t_\alpha|^2}{|t_1|^2 + \cdots + |t_n|^2}$$
    and we can pull this back through $\sigma : U \times \PP^{n-1} \rightarrow \PP^{n-1}$ to give a metric $h_1$ on $[S]^*|_{\tilde U}$. Now we extend this metric to all of $[S]^*$. Let $U' \subseteq U$ be such that $\bar U' \subset U$. Find a smooth bump function $\rho \ge 0$ such that $\rho=1$ on $U'$ and $\rho = 0$ outside $U$. Using triviality of $[S]^*|_{\tilde X \backslash U'}$ (because $S \subset U'$) give $[S^*]$ a constant metric $h_2$. We get a metric everywhere on $[S]^*$ by setting
    $$h = \rho h_1 + (1-\rho) h_2$$
    
    Consider any Hermitian metric on $K_{X}$. Then by Lemma~\ref{linebundle curvature addition},
    \begin{align*}
    \Theta_{\pi^* E^\mu \otimes [S]^* \otimes K_{\tilde X}^*} =& \mu \Theta_{\pi^* E} + \Theta_{[S]^*} + \Theta_{\pi^* K_X} + (n-1) \Theta_{[S]^*} \\
    =& \mu \Theta_{\pi^* E} + n \Theta_{[S]^*} + \Theta_{\pi^* K_X}
    \end{align*}
    First note that by positivity of $\Theta_{\pi^* E}$ away from $S$ and compactness of $X$, we can find $\mu_2$ such that when $\mu \ge \mu_1$,
    \begin{align*}
    & \mu \Theta_{\pi^* E} + \Theta_{\pi^* K_X} > 0 \text{ away from $S$} \\
    & \mu \Theta_{\pi^* E} + \Theta_{\pi^* K_X} \ge 0 \text{ everywhere}
    \end{align*}
    where we find $\mu_1$ locally using matrix expressions of $\Theta$ and take maximum over a finite subcover, and $\ge 0$ is shown by continuity. As $\Theta_{[S]^*}=0$ at $\tilde X \backslash \tilde U$, it the result is essentially proven for $\tilde X \backslash \tilde U$ and we only need to look at $\tilde U$. 
    
    At $\tilde U' = \tilde U \cap (U' \times \PP^{n-1})$, both of $\Theta_{\pi^*E}$ and $\Theta_{L^*} \cong \Theta_{\sigma^* H}$ are positive and thus any positive linear combination of them is positive too. At $\tilde U \backslash \tilde U'$, we may use matrix expression to find $\mu_2$ so that when $\mu \ge \mu_1$, $\mu \Theta_{\pi^* E} + \Theta_{L^*} > 0$ there. Now taking $\mu_0 = \mu_1 + n \mu_2$, we obtain $\mu$ so that
    $$\Theta_{\pi^* E^\mu \otimes [S]^* \otimes K_{\tilde X}^*} = \mu \Theta_{\pi^* E} + n \Theta_{[S]^*} + \Theta_{\pi^* K_X} = \begin{cases} \mu \Theta_{\pi^* E} + \Theta_{\pi^* K_X} & \text{ in $\tilde X \backslash \tilde U$} \\ \mu \Theta_{\pi^* E} + \Theta_{\sigma^* H} & \text{ in $\tilde U'$} \\ \mu \Theta_{\pi^* E} + \Theta_{[S]^*} & \text{ in $\tilde U \backslash \tilde U'$} \end{cases}$$
    which are all positive as desired.
    
    We may find $\mu$ works locally for $p$, by seeing that positivity is retained if we choose $q$ near $p$ by continuity, and using blow-up expression centered at $p$. By compactness, we can find $\mu$ that works for all $p$.
    
    Similarly we can prove the second and third statements and take the maximum of all obtained $\mu$.
\end{proof}
Note that we used Proposition~\ref{positive bundle criterion} in assuming that $E$ has a metric that induces a positive form $i\Theta_E$.

\subsection{The Kodaira Embedding Theorem}

Now we state and prove the Kodaira embedding theorem.

\begin{theorem}[Kodaira]
    A compact complex manifold $X$ is a projective variety if there is a positive line bundle on $X$.
\end{theorem}
\begin{proof}
    Let $E \rightarrow X$ be the given line bundle and let $F = E^\mu$ where $\mu$ is as in Proposition~\ref{blowup positive}. Let $\varphi = \{\varphi_1, \cdots \varphi_m\}$ be a basis for $\mc O(F,X)$ and define the following map:
    \begin{align*}
        \Phi_\varphi: & X \rightarrow \PP^m \\
        & x \mapsto [\varphi_1(f)(x): \cdots \varphi_m(f)(x) ]
    \end{align*}
    where $f$ is any frame near $x$. Note that
    $$\varphi_i(f A)(x) : \varphi_j(f A)(x) = A^{-1}(x) \varphi_i(f)(x) : A^{-1}(x) \varphi_j(f)(x) = \varphi_i(f)(x) : \varphi_j(f)(x)$$
    This almost shows well-definedness of $\Phi_\varphi$, but we also need to show that $\varphi_j$ don't vanish identically at any point. This means that the map $\forall x, \mc O(F, X) \rightarrow F_x$ is surjective; a single nonzero value mapping to $F_x$ shows that basis elements don't vanish identically. We will show that $\Phi_\varphi$ is an embedding using similar surjectivity criteria.
    
    Let $\mf m_x$ be the sheaf of germs of holomorphic functions vanishing at $x$ and similarly let $\mf{m}_{xy}$ specify vanishing at both $x,y$, with $x=y$ indicating second order vanishing. Consider the following sequence:
    \begin{align*}
        0 \rightarrow \mf m_{xy} \rightarrow \mc O \rightarrow \mc O/\mf m_{xy} \rightarrow 0
    \end{align*}
    By tensoring with locally free sheaf $\mc O(F)$, we get another exact sequence
    \begin{align*}
        0 \rightarrow \mf m_{xy} \otimes_{\mc O} \mc O(F) \rightarrow \mc O(F) \rightarrow \mc O/\mf m_{xy} \otimes_{\mc O} \mc O(F) \rightarrow 0
    \end{align*}
    As $\mc O / \mf m_{xy}$ encodes local data at $x,y$, we have
    \begin{align*}
        (\mc O/\mf m_{xy} \otimes_{\mc O} \mc O(F))(U) = \begin{cases} F_x \oplus F_y & \text{ if $x,y \in U$} \\ F_x & \text{ if $x \in U$} \\ F_y & \text{ if $y \in U$} \\ 0 & \text{ otherwise} \end{cases}
    \end{align*}
    and
    \begin{align*}
        (\mc O/\mf m_{xx} \otimes_{\mc O} \mc O(F))(U) = \begin{cases} \mc O_x / \mf m_x^2 \otimes_\CC F_x & \text{ if $x \in U$} \\ 0 & \text{ otherwise} \end{cases}
    \end{align*}
    We claim that surjectivity of $\mc O(X,F) \rightarrow ((\mc O/\mf m_{xy}) \otimes_{\mc O} \mc O(F))(X)$ implies well-definedeness, injectivity, and the rank being full. If $\mc O(X,F) \rightarrow F_x \oplus F_y$ is surjective for any distinct $x,y$, then we first get well-definedness of the map as remarked above. Also there exist $\varphi_1', \varphi_2'$ that map to $(1,0), (0,1)$ at $x,y$; by completing basis $\varphi' = \{ \varphi_j\}$ in $\mc O(X,F)$, we obtain $\Phi_{\varphi'}$ such that $\Phi_{\varphi'}(x) \neq \Phi_{\varphi'}(y)$ and by taking linear transformation in $\PP^{n-1}$ to the given basis $\varphi$, we see that $\Phi_{\varphi}(x) \neq \Phi_{\varphi'}(y)$ too. Also, in the case of $x=y$, surjectivity implies that we can find global sections $\xi_0, \xi_1, \cdots \xi_n \in \mc O(X,F)$ such that
    \begin{align*}
        & \xi_0(x) \neq 0 \\
        & \xi_j(x) = 0, \d \xi_j(x) = \d z_j \text{ for $j=1, \cdots n$}
    \end{align*}
    As $\xi_j$ are linearly independent (look at value at $x$ and cotangent space), we may complete them to a basis $\{\xi_0, \cdots \xi_{m-1} \}$ of $\mc O(X,F)$. Then Jacobian of $\Phi_\varphi$ evaluates to the coefficient of
    $$\d (\frac{\xi_1(f)}{\xi_0(f)}) \wedge \cdots \wedge \d (\frac{\xi_n(f)}{\xi_0(f)}) = \bigwedge_{j=1}^n \frac{\d \xi_j(f) \xi_0(f) }{\xi_0(f)^2} = \frac1{\xi_0(f)(x)^n} \d z_1 \wedge \cdots \d z_n$$
    which is nonzero. This gives the maximal rank.
    
    It remains to prove the two surjectivity relations. They follow directly if we could directly show that both the following cohomology are zero:
    $$H^1(X, \mc O(F) \otimes \mf m_x^2), H^1(X, \mc O(F) \otimes \mf m_{xy})$$
    This is a possible approach, but requires some work. Instead we will make use of the blow-up construction done before and approach the problem roundabout.
    
    Let the blow-up of $X$ at $x$ be $\pi: \tilde X \rightarrow X$ with $\pi^{-1}(p) = S$. Let $\mc I_S$ be the sheaf of holomorphic functions on $X$ that vanish along $S$, and let $\mc I_S^2$ be the sheaf of second-order vanishing at $S$. Denote $\tilde{\mc O} = \mc O_{\tilde X}$ and $\tilde{F} = \pi^* F$. Then we have the following exact sequence of sheaves:
    $$0 \rightarrow \mc I_S^2 \rightarrow \tilde{\mc O} \rightarrow \tilde{\mc O}/\mc I_S^2 \rightarrow 0$$
    and by tensoring the locally free sheaf $\mc O(\tilde F)$, we get another exact sequence of sheaves:
    $$0 \rightarrow \mc O(\tilde F) \otimes \mc I_S^2 \rightarrow \mc O(\tilde F) \rightarrow \mc O(\tilde F) \otimes \tilde{\mc O}/\mc I_S^2 \rightarrow 0$$
    Pullback by $\pi: \tilde X \rightarrow X$ induces a commutative diagram:
    \[\begin{tikzcd}
    0 \arrow[r] & \mc O(\tilde F) \otimes \mc I_S^2 \arrow[r] & \mc O(\tilde F) \arrow[r] & \mc O(\tilde F) \otimes \tilde{\mc O}/\mc I_S^2 \arrow[r] & 0 \\
    0 \arrow[r] & \mc O(F) \otimes \mf m_x^2 \arrow[r] \arrow[u,"\pi_1^*"] & \mc O(F) \arrow[r] \arrow[u,"\pi^*"] & \mc O(F) \otimes \mc O/\mf m_x^2 \arrow[r] \arrow[u,"\pi_2^*"] & 0
    \end{tikzcd}\]
    $\pi_1^*$ is well-defined since if a section of $\mc O(F)$ vanishes to second order at $x$, its pullback, a section of $\mc O(\tilde F)$, vanishes to second order at $S$ and thus is an element of $\mc O(\tilde F) \otimes \mc I_S^2$. This also gives well-definedenss of $\pi_2^*$ and we can see that $\pi_2^*$ is injective. 
    
    We claim that the global section maps $\pi^*(X), \pi_1^*(X)$ are isomorphisms. The inverse to $\pi^*(X)$ is constructed by pullback through $(\pi|_{\tilde X \backslash S})^{-1}$ (which is an isomorphism); this way we first obtain a holomorphic section of $F$ defined away from $x$, and by Hartogs' theorem, it extends to all of $X$. As vanishing order at $x$ transfers to vanishing order at $S$, we see that $\pi_1^*(X)$ is also isomorphism.
    
    Now if we show that $H^1(\tilde X, \tilde O(\tilde{F}))=0$, then $\mc O(X,\tilde F) \rightarrow (\mc O(\tilde F) \otimes \tilde{\mc O} / \mc I_S^2)(X)$ is surjective, and by $\pi_1^*, \pi^*$ being isomorphisms we see that $\mc O(X,F) \rightarrow (\mc O(F) \otimes \mc O/\mf m_x^2)(X)$ is also isomorphism, and the surjectivity relation will be shown.
    
    It can be shown that $\mc I_S^2 \cong \mc O(([S]^*)^2)$. Then $\mc O(\tilde F) \otimes ([S]^*)^2 \otimes K_{\tilde X}^*$ is a positive line bundle by Proposition~\ref{blowup positive}. By Kodaira-Nakano vanishing theorem (a), $H^1(\tilde X, \mc O(\tilde F)\otimes \mc O(([S]^*)^2)) = 0$ and we're done.
    
    We also have $H^1(\tilde X, \mc O(\tilde F) \otimes \mc O([S_{pq}]^*))=0$, and similarly we get the injectivity relation.
\end{proof}


\end{document}